\pdfoutput=1
\RequirePackage{ifpdf}
\ifpdf 
\documentclass[pdftex]{sigma}
\else
\documentclass{sigma}
\fi

\numberwithin{equation}{section}

\newtheorem{Theorem}{Theorem}[section]
\newtheorem{Corollary}[Theorem]{Corollary}
\newtheorem{Lemma}[Theorem]{Lemma}
\newtheorem{Proposition}[Theorem]{Proposition}
 { \theoremstyle{definition}
\newtheorem{Definition}[Theorem]{Definition}

\newtheorem{Remark}[Theorem]{Remark}
\newtheorem{Question}[Theorem]{Question} }

\usepackage[all]{xy}
\usepackage{mdwlist}
\usepackage{mathtools}
\usepackage{stmaryrd}

\newcounter{mnotecount}[section]

\newcommand{\rmnote}[1]{}

\def\al{\alpha}
\def\be{\beta}
\def\ga{\gamma}
\def\de{\delta}
\def\ep{\epsilon}

\def\et{\eta}
\def\th{\theta}

\def\la{\lambda}

\def\rh{\rho}

\def\si{\sigma}

\def\ta{\tau}

\def\vh{\varphi}

\def\ps{\psi}

\def\Ga{\Gamma}
\def\De{\Delta}
\def\Th{\Theta}

\def\Si{\Sigma}

\def\Om{\Omega}

\def\C{\mathbb{C}}

\def\K{\mathbb{K}}
\def\N{\mathbb{N}}

\def\R{\mathbb{R}}
\def\Z{\mathbb{Z}}

\def\cA{\mathcal{A}}
\def\cB{\mathcal{B}}
\def\cC{\mathcal{C}}

\def\cI{\mathcal{I}}
\def\cJ{\mathcal{J}}

\def\cL{\mathcal{L}}

\def\p{\partial}

\def\oI{\overline I}
\def\ol{\overline}
\def\ul{\underline}

\def\Hoeld{\on{H\ddot{o}ld}}
\def\Lip{\on{Lip}}

\def\const{M}

\def\we{\hat}

\renewcommand{\Re}{\mathrm{Re}}
\renewcommand{\Im}{\mathrm{Im}}
\def\<{\langle}
\def\>{\rangle}
\renewcommand{\o}{\circ}
\def\cq{{/\!\!/}}
\def\acts{\circlearrowleft}

\let\on=\operatorname
\newcommand{\sr}[1]%
{\ifmmode{}^\dagger\else${}^\dagger$\fi\ifvmode
\vbox to 0pt{\vss
 \hbox to 0pt{\hskip\hsize\hskip1em
 \vbox{\hsize3cm\raggedright\pretolerance10000
 \noindent #1\hfill}\hss}\vss}\else
 \vadjust{\vbox to0pt{\vss%
 \hbox to 0pt{\hskip\hsize\hskip1em%
 \vbox{\hsize3cm\raggedright\pretolerance10000%
 \noindent #1\hfill}\hss}\vss}}\fi%
}

\begin{document}

\newcommand{\arXivNumber}{2003.01967}

\renewcommand{\PaperNumber}{037}

\FirstPageHeading

\ShortArticleName{Sobolev Lifting over Invariants}

\ArticleName{Sobolev Lifting over Invariants}

\Author{Adam PARUSI\'NSKI~$^{\rm a}$ and Armin RAINER~$^{\rm b}$}

\AuthorNameForHeading{A.~Parusi\'nski and A.~Rainer}

\Address{$^{\rm a)}$~Universit\'e C\^ote d'Azur, CNRS, LJAD, UMR 7351, 06108 Nice, France}
\EmailD{\href{mailto:adam.parusinski@univ-cotedazur.fr}{adam.parusinski@univ-cotedazur.fr}}

\Address{$^{\rm b)}$~Fakult\"at f\"ur Mathematik, Universit\"at Wien, Oskar-Morgenstern-Platz~1,\\
\hphantom{$^{\rm b)}$}~A-1090 Wien, Austria}
\EmailD{\href{mailto:armin.rainer@univie.ac.at}{armin.rainer@univie.ac.at}}

\ArticleDates{Received November 04, 2020, in final form March 29, 2021; Published online April 10, 2021}

\Abstract{We prove lifting theorems for complex representations $V$ of finite groups $G$. Let $\si=(\si_1,\dots,\si_n)$ be a minimal system of homogeneous basic invariants and let $d$ be their maximal degree. We~prove that any continuous map $\ol f \colon \R^m \to V$ such that $f = \si \o \ol f$ is of class $C^{d-1,1}$
is locally of Sobolev class $W^{1,p}$ for all $1 \le p < d/(d-1)$. In~the case $m=1$ there always exists a continuous choice $\ol f$ for given $f\colon \R \to \si(V) \subseteq \C^n$. We~give uniform bounds for the $W^{1,p}$-norm of $\ol f$ in terms of the $C^{d-1,1}$-norm of $f$. The result is optimal: in general a lifting $\ol f$ cannot have a higher Sobolev regularity and it even might not have bounded variation if $f$ is in a larger H\"older class.}


\Classification{22E45; 26A16; 46E35; 14L24}

\vspace{-3mm}
\section{Introduction}\vspace{-1mm}

\subsection{Motivation and introduction to the problem}

This paper arose from our wish to understand and extend the principles behind our proof of the optimal Sobolev regularity of roots of smooth families of polynomials \cite{ParusinskiRainerHyp,ParusinskiRainerAC,ParusinskiRainer15,Parusinski:2020aa}. Here we look at this problem from a representation theoretic view point. In~fact, choosing the roots of~a~family of monic polynomials\vspace{-.5ex}
\begin{gather*}
 P_{a(x)}(Z) = Z^n + \sum_{j=1}^n a_j(x) Z^{n-j}\vspace{-1ex}
\end{gather*}
means solving the system of equations
 \begin{gather*}
 a_1(x) = \sum_{j=1}^n \la_j(x), \\
 a_2(x) = \sum_{1 \le j_1 <j_2 \le n} \la_{j_1}(x) \la_{j_2}(x),\\[-1ex]
\cdots\cdots\cdots\cdots\cdots\cdots\cdots\cdots\cdots\cdots \\[-1.5ex]
 a_n(x) = \prod_{j=1}^n \la_j(x)\vspace{-1ex}
 \end{gather*}
for functions $\la_j$, $ j = 1,\dots,n$. In~other words, it means lifting the map $a = (a_1,\dots,a_n)$ over the map
$\si=(\si_1,\dots,\si_n)$ the components of which are the elementary symmetric functions in~$n$ variables,
\begin{gather*}
 \si_i(X_1,\dots,X_n) = \sum_{1 \le j_1 < \cdots < j_i \le n} X_{j_1} X_{j_2} \cdots X_{j_i}.
\end{gather*}

The map $\sigma$ can be identified with the orbit projection of the tautological representation of~the symmetric group $\on{S}_n$
on $\C^n$ (it acts by permuting the coordinates).

In this paper we shall solve the generalized problem for complex finite-dimensional representations of finite groups.
Let $G$ be a finite group.
Let $\rh\colon G \to \on{GL}(V)$ be a representation of~$G$ on~a~finite-dimensional complex vector space $V$. By~Hilbert's finiteness theorem the algebra of~inva\-riant polynomials $\C[V]^G$ is finitely generated.
Let $\si_1,\dots,\si_n$ be a system of generators, we~call them \emph{basic invariants}, and let
$\si= (\si_1,\dots,\si_n)$ be the resulting map $\si \colon V \to \C^n$.
The map~$\si$ separates $G$-orbits and hence induces a homeomorphism between the orbit space $V/G$ and the image $\si(V)$.
(Notice that since $G$ is finite and thus all $G$-orbits are closed, there is a bijection between the orbits and
the points in the affine variety $V \cq G$ with coordinate ring $\C[V]^G$; in other words the \emph{categorical quotient}
$V \cq G$ is a \emph{geometric quotient}.)
As a consequence we may identify $V/G$ with $\si(V)$ and the canonical orbit projection $V \to V/G$ with $\si \colon V \to \si(V)$. We~will also write $G \acts V$ for the representation $\rh$.

The basic invariants can be chosen to be homogeneous polynomials. A system of homogeneous basic invariants is \emph{minimal} if
none among them is superfluous. In~that case their number and their degrees are uniquely determined (cf.\ \cite[p.~95]{DK02}).

Assume that a map $f \colon \Om \to \si(V)$ defined on some open subset $\Om \subseteq \R^m$ is given. We~will assume that $f$ possesses some degree of differentiability as a map into $\C^n$.
The question we will address in this paper is the following:
\begin{quote}
 How \emph{differentiable} can lifts of $f$ over $\si$ be? By a \emph{lift of $f$ over $\si$} we mean a map
 $\ol f \colon \Om \to V$ such that $f = \si \o \ol f$.
\end{quote}
Simple examples show that, in general, a big loss of regularity occurs from $f$ to lifts of $f$. We~will determine the optimal regularity of lifts among the Sobolev spaces $W^{1,p}$ under minimal differentiability
requirements on $f$. In~particular, the optimal $p>1$ will be determined as an~explicit function of the maximal homogeneity degree
of the basic invariants.

Note that the results do not depend on the choice of the basic invariants since any two choices differ by a polynomial
diffeomorphism.

Our results could be useful in connection with the orbit space reduction of equivariant dyna\-mical system
for lifting the solutions from orbit space (even though it is not clear when a lifted solution solves the
original differential equation).
Another application to multi-valued Sobolev functions is discussed at the end of the paper.

\subsection{The main results}

The first result concerns the lifting of curves. We~recall that, since $G$ is finite, each continuous $a \colon I \to \si(V)$, where $I\subseteq \R$ is an interval,
has a continuous lift $\ol a \colon I \to V$, by \cite[Theorem 5.1]{LMRac}.

\begin{Theorem}\label{main}
 Let $G$ be a finite group and let $G \acts V$ be a representation of~$G$ on a finite-dimensional
 complex vector space $V$. Let $\si=(\si_1,\dots,\si_n)$ be a (minimal) system of homogeneous
 basic invariants of degrees $d_1,\dots,d_n$ and set $d = \max_i d_i$.
 Let $a \in C^{d-1,1}([\al,\be],\si(V))$ be a~curve
 defined on an open bounded interval
 $(\al,\be)$ with values in $\si(V)$.
 Then each continuous lift $\overline a \colon (\al,\be) \to V$ of $a$ over $\si$ is absolutely continuous and
 belongs to $W^{1,p}((\al,\be),V)$
 with
 \begin{gather} \label{eq:main}
 	 \|\ol a'\|_{L^p((\al,\be))} \le
 C(G \acts V,(\be-\al),p)\, \max_{1 \le j \le n} \|a_j\|^{1/d_j}_{C^{d-1,1}([\al, \be])}		
 	\end{gather}	
 for all $1 \le p < d/(d-1)$,
 where $C$ is a constant which depends only on the representation $G \acts V$, the length of the interval $(\al,\be)$, and $p$.	
\end{Theorem}

{\samepage
The conclusion of the theorem is in general optimal among Sobolev spaces, the differentiability assumption on $a$ is best possible;
see Remark~\ref{optimality}.
Here and below we use the notation
\begin{gather*}
 C^{d-1,1}([\al,\be],\si(V)) := C^{d-1,1}([\al,\be],\C^n) \cap \si(V)^{(\al,\be)},
\end{gather*}
the H\"older class $C^{d-1,1}$ is defined in Section~\ref{functionspaces}.

}

\begin{Remark}\qquad
\label{rem:main}

 ($a$) In general the constant in~\eqref{eq:main} is of the form
 \begin{gather*}
 C(G \acts V,p)\, \max\big\{1, (\be-\al)^{1/p}, (\be-\al)^{-1+1/p} \big\}.
 \end{gather*}

 ($b$) If the curve $a$ starts, ends, or passes through $0$ (that is the most singular point in $\si(V)$),
 then the constant in~\eqref{eq:main} is of the form
 \begin{gather} \label{eq:betterbound}
 C(G \acts V,p)\, \max\big\{1, (\be-\al)^{1/p} \big\}.
 \end{gather}

 ($c$) If the representation is coregular, then for all $a$ satisfying the assumptions of
 Theorem~\ref{main} the constant is of the form~\eqref{eq:betterbound}.
 A representation $G \acts V$ is called \emph{coregular} if $\C[V]^G$ is~isomorphic to a polynomial algebra, i.e.,
 there is a system of basic invariants without polynomial relations among them. By~the Shephard--Todd--Chevalley theorem~\cite{Chevalley:1955aa, Serre:1968aa, Shephard:1954aa},
 this is the case if~and only if $G$ is generated by pseudoreflections.

 ($d$) The constant is also of the form~\eqref{eq:betterbound}
 if the curve $a$ satisfies $a^{(j)}(\al) = a^{(j)}(\be)=0$ for all $j=1,\dots,d-1$.
\end{Remark}

\begin{Question}
	The constant in~\eqref{eq:main} tends to infinity as $p \to d/(d-1)=:d'$.
	Our proof yields that it blows up like a power of $(d'-p)^{-1/p}$, since we have to iterate the
	inequality~\eqref{eq:qp} several times when we pass from $L^{d'}_w$-(quasi)norm to $L^p$-norm.
	This is necessary, since the former is not $\si$-additive.
	We expect that the asymptotic behavior of the constant as $p \to d'$ is actually better:
	{\it Is the constant actually $O\big((d'-p)^{-1/p}\big)$ as $p \to d'$?
	Can one replace the $L^p$-norm of $\ol a'$ by the $L^{d'}_w$-$($quasi$)$norm in~\eqref{eq:main}?} 	
\end{Question}

The lifting of mappings defined in open domains of dimension $m > 1$ essentially
admits the same regularity as for curves, provided that continuous lifting is possible.
However, there are well-known topological obstructions for continuous lifting in general. We~will prove the following

\begin{Theorem} \label{main2}
	In the setting of Theorem~$\ref{main}$
	let $f \in C^{d-1,1}\big(\ol \Om,\si(V)\big)$,
	where $\Om \subseteq \R^m$ is an~open bounded box $\Om = I_1 \times \cdots \times I_m$.
	Then each continuous lift $\ol f \colon U \to V$ of $f$ over $\si$
	defined on an~open subset $U \subseteq \Om$
	belongs to $W^{1,p}(U,V)$ for all $1 \le p < d/(d-1)$
 and satisfies
 \begin{gather} \label{eq:main2}
 \big\|\nabla \ol f\big\|_{L^p(U)} \le
 C(G \acts V,\Om,m,p)\, \max_{1 \le j \le n} \|f_j\|^{1/d_j}_{C^{d-1,1}(\ol \Om)}
 \end{gather}
 for all $1 \le p < d/(d-1)$,
 where $C$ is a constant which depends only on the representation $G \acts V$, on $\Om$, $m$, and $p$.
\end{Theorem}

The case $U = \Om$ is not excluded!
It is clear that Theorem~\ref{main2} implies a version of the statement, where
$\Om \subseteq \R^m$ is any bounded open set,
$U \Subset \Om$ is relatively compact open in $\Om$,
and the constant also depends on $U$ (or more precisely on a cover of $U$ by boxes contained in $\Om$).
Concerning a global result we have the following

\begin{Remark} \label{rem:main2}
	If $G \acts V$ is coregular, then Theorem~\ref{main2} holds as stated for any bounded Lipschitz domain $\Om$.
\end{Remark}

When continuous lifting is impossible,
we expect that a general $BV$-lifting result is true analogous to the existence of $BV$-roots
for smooth polynomials proved in \cite{Parusinski:2020aa}. We~shall not pursue that question in this paper.

\subsection{Linearly reductive groups}

An algebraic group $G$ is called \emph{linearly reductive} if for each rational representation
$V$ and each subrepresentation $W \subseteq V$ there is a subrepresentation $W' \subseteq V$ such that
$V = W \oplus W'$.

For rational representations of linearly reductive groups $G$ Hilbert's finiteness theorem is true,
that is the algebra of~$G$-invariant polynomials $\C[V]^G$ is finitely generated. Let $\si=(\si_1,\dots, \si_n)$ be a system of generators.
Then the map $\si \colon V \to \si(V) \subseteq \C^n$ can be identified with the morphism $V \to V\cq G$ induced by the
inclusion $\C[V]^G \to \C[V]$; the \emph{categorical quotient} $V\cq G$ is the affine variety with
coordinate ring $\C[V]^G$. In~general $V\cq G$ is not a \emph{geometric quotient}, that is the
$G$-orbits in $V$ are not in a one-to-one correspondence with the points in $V\cq G$. In~fact,
for~every point $z \in V\cq G$ there is a unique closed orbit in the fiber $\si^{-1}(z)$ which lies in the
closure of every other orbit in this fiber.

In this setting it is not clear if a continuous curve in $\si(V)$ admits a continuous lift to $V$. The~notion of \emph{stability} in geometric invariant theory provides a remedy.
A point $v \in V$ is called \emph{stable} if the orbit $Gv$ is closed and the isotropy group
$G_v = \{g \in G \colon gv=v\}$ is finite.
The~set $V^s$ of~stable points in $V$ is $G$-invariant and open in $V$,
and its image $\si(V^s)$ is open in~$V\cq G \cong \si(V)$ (cf.\ \cite[Proposition 5.15]{Mukai:2003aa}).
The restriction $\si \colon V^s \to \si(V^s)$ of the map $\si$ provides a one-to-one correspondence
between points in $\si(V^s) \cong V^s/G$ and $G$-orbits in $V^s$, that is $V^s/G$ is a geometric quotient.

\begin{Lemma}
 Let $a \colon I \to \si(V^s)$, where $I \subseteq \R$ is an open interval, be continuous.
 Then $a$ has a~con\-tinuous lift $\ol a \colon I \to V^s$.
\end{Lemma}

\begin{proof}
 For every $v \in \si^{-1}(a(I))$ there is a local continuous lift $\ol a_v$ of $a$ defined on some
 open subinterval $I_v$ of $I$ with $\ol a_v(t_v) = v$ for some point $t_v \in I_v$. This follows from the
 lifting theorem \cite[Theorem 5.1]{LMRac}, since locally at any $v$ the problem can be reduced to the
 slice representation of~the isotropy group $G_v$ which is finite (cf.\ Theorem~\ref{Luna}).
 Now each continuous lift $\ol a$ of~$a$ defined on a proper subinterval $J$ of $I$ has an extension to a~larger interval $J' \subseteq I$. Thus there is a continuous lift on~$I$.
 Indeed, say the right endpoint $t_1$ of $J$ lies in $I$.
 There is continuous lift $\ol a_v \colon I_v \to V^s$ for $v \in \si^{-1}(a(t_1))$.
 Choose $t_0 \in J \cap I_v$ and $g \in G$ such that $\ol a(t_0) = g \ol a_v(t_0)$.
 Then~$g \ol a_v$ extends the continuous lift $\ol a$ beyond $t_1$.
\end{proof}

As a corollary of Theorem~\ref{main} we obtain

\begin{Theorem}
 Let $G$ be a linearly reductive group and let $G \acts V$ be a rational representation of~$G$ on a finite-dimensional
 complex vector space $V$. Let $\si=(\si_1,\dots,\si_n)$ be a (minimal) system of homogeneous
 basic invariants of degrees $d_1,\dots,d_n$ and set $d = \max_i d_i$.
 Let $a \in C^{d-1,1}([\al,\be],\si(V^s))$ be a curve defined on a compact interval with
 $a([\al,\be]) \subseteq \si(V^s)$.
 Then there exists an absolutely continuous lift
 $\overline a \colon [\al,\be] \to V^s$ of $a$ over $\si$ which
 belongs to $W^{1,p}([\al,\be],V^s)$
 with
 \begin{gather} \label{eq:stable}
 \|\ol a'\|_{L^p([\al,\be])} \le
 C(G \acts V,[\al,\be],p)\, \max_{1 \le j \le n} \|a_j\|^{1/d_j}_{C^{d-1,1}([\al, \be])}
 \end{gather}
 for all $1 \le p < d/(d-1)$.
\end{Theorem}

\begin{proof}
 Since the lifting problem can be reduced to the slice representations (cf.\ Theorem~\ref{Luna} and Lemma~\ref{lem:e}), and for all $v \in V^s$
 the isotropy group $G_v$ is finite, Theorem~\ref{main} implies that for~all~$v \in \si^{-1}(a([\al,\be])$
 there exists a local absolutely continuous lift $\ol a_v$ of $a$ defined on a~sub\-in\-ter\-val $I_v$ of $[\al,\be]$ which is open in the relative topology on $[\al,\be]$ such that
 \begin{gather*}
 \|\ol a_v'\|_{L^p(I_v)} \le
 C(G \acts V,|I_v|,p)\, \max_{1 \le j \le n} \|a_j\|^{1/d_j}_{C^{d-1,1}(\ol I_v)}, \qquad 1
 \le p < \frac{d}{d-1},
 \end{gather*}
 and there is a point $t_v \in I_v$ with $\ol a_v(t_v) = v$. By~compactness, there is a finite collection of local lifts which cover $[\al,\be]$.
 It is then easy to glue these pieces (after applying fixed transformations from $G$)
 to an absolutely continuous lift $\ol a$ defined on $[\al,\be]$ and satisfying~\eqref{eq:stable}.
\end{proof}

\looseness=1
For a mapping $f$ defined on a compact subset $K$ of $\R^m$ with $f(K) \subseteq \si(V^s)$ the situation is
more complicated. We~can apply Theorem~\ref{main2} to the slice representations at any point $v \in V^s$.
But it is not clear if these local (and partial) lifts can be glued together in a continuous fashion.

\subsection{Polar representations}
More can be said for polar representations (which include e.g.\ the adjoint actions). The following results can be found in \cite{DK85}.
Let $G$ be a linearly reductive group and let $G \acts V$ be a representation of~$G$ on a finite-dimensional
complex vector space $V$.
Let $v \in V$ be such that $Gv$ is closed and consider the linear subspace
$\Si_v =\{x \in V \colon \mathfrak g x \subseteq \mathfrak g v\}$, where $\mathfrak g$ denotes the Lie algebra of~$G$.
All~orbits that intersect $\Si_v$ are closed, whence $\dim \Si_v \le \dim V \cq G$.
The representation $G \acts V$ is said to be \emph{polar} if there exists $v \in V$ with closed orbit $Gv$
and $\dim \Si_v = \dim V \cq G$. Then $\Si_v$ is called a \emph{Cartan subspace} of $V$.
Any two Cartan subspaces are $G$-conjugate. Let us fix one Cartan space $\Si$.
All closed orbits in $V$ intersect $\Si$.

The \emph{Weyl group} $W$ is defined by $W = N_G(\Si)/Z_G(\Si)$, where $N_G(\Si)= \{ g \in G \colon g \Si = \Si\}$ is~the
normalizer and $Z_G(\Si)= \{ g \in G \colon g x = x \text{ for all }x \in\Si\}$ is the centralizer of $\Si$ in $G$.
The~Weyl group is finite and the intersection of any closed $G$-orbit in $V$ with the Cartan subspace is precisely
one $W$-orbit.
The ring $\C[V]^G$ is isomorphic via restriction to the ring $\C[\Si]^W$.
If $G$ is connected, then $W$ is a pseudoreflection group and hence $\C[V]^G \cong \C[\Si]^W$ is a polynomial ring,
by the Shephard--Todd--Chevalley theorem~\cite{Chevalley:1955aa, Serre:1968aa, Shephard:1954aa}.

\begin{Theorem}
	Let $G \acts V$ be a polar representation of a linearly reductive group $G$.
	Let $\si=(\si_1,\dots,\si_n)$ be a (minimal) system of homogeneous
 basic invariants of degrees $d_1,\dots,d_n$ and set $d = \max_i d_i$.
\begin{enumerate}\itemsep=0pt
 	\item[$1.$] Let $a \in C^{d-1,1}([\al,\be],\si(V))$ be a curve
 defined on an open bounded interval
 $(\al,\be)$ with values in $\si(V)$. Then there exists an absolutely continuous lift $\ol a \colon (\al,\be) \to V$ of $a$
 over $\si$ which belongs to $W^{1,p}((\al,\be),V)$ for all $1 \le p <d/(d-1)$ and satisfies~\eqref{eq:main}.
 	\item[$2.$] Let $f \in C^{d-1,1}(\ol \Om,\si(V))$,
	where $\Om \subseteq \R^m$ is an open bounded box $\Om = I_1 \times \cdots \times I_m$.
	Each continuous lift $\ol f$ defined in an open subset $U\subseteq \Om$ with values in a Cartan subspace $\Si$
	is of class $W^{1,p}$ on $U$ for all $1 \le p <d/(d-1)$ and satisfies~\eqref{eq:main2}.
	\item[$3.$] In the case that $G$ is connected the constant in~\eqref{eq:main} is of the form~\eqref{eq:betterbound}
	and $\Om$ can be any bounded Lipschitz domain.
 \end{enumerate}
\end{Theorem}

\begin{proof}	Apply Theorems~\ref{main} and~\ref{main2} to the Weyl group $W$ acting on a Cartan sub\-space~$\Si$.
	If~$G$ is connected, then $W \acts \Si$ is coregular, so (3) follows from Remarks~\ref{rem:main} and~\ref{rem:main2}.
\end{proof}

\subsection{A related problem}

In an analogous way one may consider the case that $V$ is a \emph{real} finite-dimensional vector space
and $\rh \colon G \to \on{O}(V)$ is an orthogonal representation of a finite group.
Again the algebra of~$G$-invariant polynomials $\R[V]^G$ is finitely generated, and a system of basic invariants $\si$ allows us to
identify $\si(V)$ with the orbit space $V/G$. In~this case $\si(V)$ is a semialgebraic subset of $\R^n$. In~that setting the problem
was solved in \cite{ParusinskiRainer14}:

\begin{Theorem} \label{real}
Let $G$ be a finite group and let $G \acts V$ be an orthogonal representation of~$G$ on a finite-dimensional
 real vector space $V$. Let $\si=(\si_1,\dots,\si_n)$ be a (minimal) system of~homogeneous
 basic invariants of degrees $d_1,\dots,d_n$ and set $d = \max_i d_i$.
\begin{enumerate}\itemsep=0pt
 \item[$1.$] Let $a \in C^{d-1,1}([\al,\be],\si(V))$.
 Then each continuous lift $\overline a \colon (\al,\be) \to V$ of $a$ over $\si$ belongs to $W^{1,\infty}((\al,\be),V)$ with
 \begin{gather*}
 \|\ol a'\|_{L^\infty((\al,\be))} \le
 C(G \acts V,(\be-\al))\, \max_{1 \le j \le n} \|a_j\|^{1/d_j}_{C^{d-1,1}([\al, \be])}.
 \end{gather*}
 Every continuous curve in $\si(V)$ has a continuous lift.
 \item[$2.$] Let $f \in C^{d-1,1}(\ol \Om,\si(V))$,
 where $\Om \subseteq \R^m$ is open and bounded.
 Then each continuous lift $\ol f \colon U \to V$ of $f$ over $\si$ defined on an open subset
 $U \subseteq \Om$ belongs to $W^{1,\infty}(U,V)$ with
 \begin{gather*}
 \|\nabla \ol f\|_{L^\infty(U)} \le
 C(G \acts V,\Om,U,m)\, \max_{1 \le j \le n} \|a_j\|^{1/d_j}_{C^{d-1,1}(\ol \Om)}.
 \end{gather*}
 \end{enumerate}
\end{Theorem}

In the special case of the tautological representation of $\on{S}_n$ on $\R^n$ this corresponds to the problem
of choosing the roots of \emph{hyperbolic} polynomials, i.e., monic polynomials all roots of~which are real;
see \cite{ParusinskiRainerHyp}.

The main difference between the complex and the real problem is that in the latter case
the map $v \mapsto \langle v, v \rangle = \|v\|^2$ is an invariant polynomial which may be taken without loss of generality
as a basic invariant and thus as a component of the map $\si$. The key is that this basic invariant
dominates all the others, by homogeneity,
\begin{gather*}
 |\si_j(v)| \le \max_{\|w\|=1} |\si_j(w)| \, \|v\|^{d_j}.
\end{gather*}
Even though we can always choose an invariant Hermitian inner product in
the complex case (by averaging over $G$) and hence assume that the representation is unitary, the
invariant form $v \mapsto \|v\|^2$ is not a member of $\C[V]^G$.
The fact that there is no invariant that dominates all others makes the complex case much more difficult.

\subsection{Elements of the proof}

We briefly describe the strategy of the proof of Theorem~\ref{main}.

The basic building block of the proof is that the result holds for finite rotation groups $C_d$ in $\C$,
where $\C[\C]^{C_d}$ is generated by $z \mapsto z^d$ and a lift of a map $f$ is a solution of the equation
$z^d = f$. This follows from \cite{GhisiGobbino13}.
Among all representations of finite groups $G$ of order $|G|$ it is the one
with the worst loss of
regularity, since in general $d \le |G|$, by Noether's degree bound, and equality can only happen for
cyclic groups (see Section~\ref{radicals}).

In the general case we first observe that evidently one may reduce to the case that the linear subspace
$V^G$ of invariant vectors is trivial. Then
Luna's slice theorem (see Theorem~\ref{Luna}) allows us to reduce the problem locally to the
slice representation $G_v \acts N_v$ of the isotropy group $G_v = \{g \in G \colon gv=v\}$ on $N_v$, where
$T_v V \cong T_v (Gv) \oplus N_v$ is a $G_v$-splitting. Since in our case $G$ is finite, we have $N_v \cong V$.
The assumption $V^G = \{0\}$ entails that for all $v \in V \setminus \{0\}$ the isotropy group $G_v$
is a proper subgroup of~$G$ which suggests to use induction.

For this induction scheme to work we need that the slice reduction is uniform in the sense that it
does not depend on the parameter $t$ of the curve $a$ in $\si(V) \subseteq \C^n$. We~achieve this by~considering the curve
\begin{gather*}
 \ul a = \big(a_k^{-d_1/d_k} a_1,\dots, a_k^{-d_n/d_k} a_n\big), \qquad \text{when}\quad a_k \ne 0,
\end{gather*}
and the compactness of the set of all $\ul a \in \si(V)$ such that $|\ul a_j|\le 1$ for all $j = 1,\dots,n$
and~\mbox{$\ul a_k = 1$}.
Let us emphasize that hereby we use a fixed continuous selection $\we a_k$ of the multi-valued function~$a_k^{1/d_k}$ which is absolutely continuous by the result for the rotation group $C_{d_k} \acts \C$.

If $a \in C^{d-1,1}([\al,\be],\si(V))$ and $t_0 \in (\al,\be)$ is such that $a(t_0) \ne 0$, then we
choose $k \in \{1,\dots,n\}$ \emph{dominant} in the sense that
\begin{gather*}
\big|a_k^{1/d_k}(t_0)\big| = \max_{1 \le j \le n} \big|a_j^{1/d_j}(t_0)\big| \ne 0.
\end{gather*}
It is easy to extend the lifts to the points, where $a$ vanishes, so we will not discuss them here. We~work on a small interval $I$ containing $t_0$ such that
for all
$j = 1,\dots,n$ and $ s = 1,\dots,d-1$,
 \begin{gather*}
 \big\|a_j^{(s)} \big\|_{L^\infty(I)} \le C(d) |I|^{-s} |a_k (t_0)|^{d_j/d_k},
 \\
 \Lip_I \big(a_j^{(d-1)}\big) \le C(d) |I|^{-d} |a_k (t_0)|^{d_j/d_k}.
 \end{gather*}
This can be achieved by choosing the interval $I$ in such a way that $t_0 \in I \subseteq (\al,\be)$ and
\begin{gather*}
M |I| + \sum_{j=1}^n \big\|\big(a_j^{1/d_j}\big)'\big\|_{L^1(I)} \le B |a_k(t_0)|^{1/d_k},
\end{gather*}
where $B$ is a suitable constant which depends only on the representation and the constant $M$ depends on
the representation and the curve $a$. Notice that here we use again absolute continuity of~radicals (i.e., the result for
complex rotation groups).
Uniform slice reduction allows us to switch to a reduced curve $b \colon I \to \ta(W)$ of class $C^{d-1,1}$,
where $H \acts W$ is a slice representation of~$G \acts V$ and the map $\ta = (\ta_1,\dots,\ta_m)$ consists
of a system of homogeneous generators for~$\C[W]^H$.
For~convenience we will refer to the tuple $(a,I,t_0,k;b)$ as \emph{reduced admissible data} for~$G \acts V$.

The core of the proof (see Proposition~\ref{induction}) is to show that, if $(a,I,t_0,k;b)$ is reduced admissible data for $G \acts V$, then
every continuous lift $\ol b \colon I \to W$ of $b$ is absolutely continuous and satisfies
\begin{gather*}
 \big\|\ol b'\big\|_{L^p(I)} \le C(d,p)\, \bigg( \big\||I|^{-1} |a_k(t_0)|^{1/d_k}\big\|_{L^p(I)} + \sum_{i=1}^m \big\|(b_i^{1/e_i})'\big\|_{L^p(I)}\bigg)
\end{gather*}
for all $1 \le p < d/(d-1)$, where $e_i = \deg \ta_i$. This is done by induction on the group order
and involves showing that the set of points $t$ in $I$, where $b(t) \ne 0$, can be covered by a special countable
collection of intervals on which $b$ defines reduced admissible data for $H \acts W$.
The~difficult part is to assure that each point is covered by
at most two intervals in the collection (see Proposition~\ref{cover})
which is needed for gluing the local $L^p$-estimates to a global estimate on~$I$.
It would suffice that each point lies in no more than a uniform finite number of intervals,
but the crucial thing is that the intervals must not be shrunk (see Remark~\ref{rem:cover}).

\subsection[An application: Q-valued functions]{An application: ${\boldsymbol Q}$-valued functions}

In Section~\ref{sec:Q-valued} we explore an interesting connection between invariant theory and the theory of~$Q$-valued functions. These are functions with values in the metric space of unordered $Q$-tuples of points in
$\R^n$ (or $\C^n$). There is a natural one-to-one correspondence between unordered $Q$-tuples of points in
$\K^n$ (where $\K$ stands for $\R$ or $\C$) and the $n$-fold direct sum of the tautological representation of the symmetric
group $S_Q$ on $\K^Q$.
Using the theory of $Q$-valued Sobolev functions rooted in variational calculus,
cf.\ \cite{Almgren00} and \cite{De-LellisSpadaro11}, we will show that our main results
entail optimal \emph{multi-valued} Sobolev lifting theorems. Thanks to the multi-valuedness there are
no topological obstructions for continuity.


\section{Function spaces}\label{functionspaces}

In this section we fix notation for function spaces and recall well-known facts.

\subsection{H\"older spaces}

Let $\Om \subseteq \R^n$ be open and bounded. We~denote by $C^0(\Om)$ the space of continuous complex valued functions on $\Om$.
For $k \in \N \cup \{\infty\}$ (and multi-indices $\ga$) we set
\begin{gather*}
 C^k(\Om)= \big\{f \in \C^\Om \colon \p^\ga f \in C^0(\Om), 0 \le |\ga| \le k\big\},
 \\
 C^k(\overline \Om) = \big\{f \in C^k(\Om) \colon \p^\ga f \text{ has a continuous extension to } \overline \Om, 0 \le |\ga| \le k\big\}.
\end{gather*}
For $\al \in (0,1]$ a function $f \colon \Om \to \C$ belongs to $C^{0,\al}(\overline \Om)$ if it is \emph{$\al$-H\"older continuous}
in $\Om$, i.e.,
\begin{gather*}
\Hoeld_{\al,\Om}(f) := \sup_{x,y \in \Om,\, x \ne y} \frac{|f(x)-f(y)|}{|x-y|^\al} < \infty.
\end{gather*}
If $f$ is \emph{Lipschitz}, i.e., $f \in C^{0,1}\big(\ol \Om\big)$, we write
$\Lip_\Om (f) :=\Hoeld_{1,\Om}(f)$. We~define
\begin{gather*}
 C^{k,\al}\big(\overline \Om\big) =\big \{f \in C^k\big(\overline \Om\big) \colon \p^\ga f \in C^{0,\al}\big(\overline \Om\big), |\ga|\le k\big\},
\end{gather*}
which is a Banach space when provided with the norm
\begin{gather*}
\|f\|_{C^{k,\al}(\overline \Om)}
:= \max_{|\ga| \le k} \sup_{x \in \Om} \big|\p^\ga f(x)\big| + \max_{|\ga|=k} \Hoeld_{\al,\Om}\big(\p^\ga f\big).
\end{gather*}

\subsection{Lebesgue spaces and weak Lebesgue spaces} 

Let $\Om \subseteq \R^n$ be open and $1 \le p \le \infty$.
Then $L^p(\Om)$ is the Lebesgue space with respect to the $n$-dimensional Lebesgue measure $\cL^n$.
For Lebesgue measurable sets $E \subseteq \R^n$ we denote by
\begin{gather*}
 |E| = \cL^n(E)
\end{gather*}
the $n$-dimensional Lebesgue measure of $E$.
Let $p' := p/(p-1)$ denote the conjugate exponent of $p$
with the convention $1' := \infty$ and $\infty' :=1$.

Let $1 \le p < \infty$ and let us assume that $\Om$ is bounded.
The \emph{weak $L^p$-space} $L_w^p(\Om)$ is the space of all measurable functions $f \colon \Om \to \C$ such that
\begin{gather*}
\|f\|_{p,w,\Om} := \sup_{r> 0} \big( r\, |\{x \in \Om \colon |f(x)| > r\}|^{1/p} \big) < \infty.
\end{gather*}
It will be convenient to \emph{normalize}:
\begin{gather*}
 \|f\|^*_{L^p(\Om)} := |\Om|^{-1/p} \|f\|_{L^p(\Om)},\\
 \|f\|^*_{p,w,\Om} := |\Om|^{-1/p} \|f\|_{p,w,\Om}.
\end{gather*}
Note that $\|1\|^*_{L^p(\Om)} = \|1\|^*_{p,w,\Om} =1$.
For $1 \le q < p < \infty$ we have (cf.\ \cite[Exercise\ 1.1.11]{Grafakos08})
\begin{gather}
 \|f\|^*_{L^q(\Om)} \le \|f\|^*_{L^p(\Om)}, \notag \\
 \|f\|^*_{q,w,\Om} \le \|f\|^*_{L^q(\Om)} \le \bigg(\frac{p}{p-q}\bigg)^{1/q}
 \|f\|^*_{p,w,\Om} \label{eq:qp}
\end{gather}
and hence
$L^p(\Om) \subseteq L_w^p(\Om) \subseteq L^q(\Om) \subseteq L_w^q(\Om)$
with strict inclusions.

We remark that $\|\cdot\|_{p,w,\Om}$ is only a quasinorm: the triangle inequality fails, but for
$f_j \in L_w^p(\Om)$
we still have
\begin{gather*}
\bigg\|\sum_{j=1}^m f_j \bigg\|_{p,w,\Om} \le m \sum_{j=1}^m \|f_j\|_{p,w,\Om}.
\end{gather*}
There exists a norm equivalent to $\|\cdot\|_{p,w,\Om}$ which makes $L_w^p(\Om)$ into a Banach space if $p>1$.

The $L^p_w$-quasinorm is $\si$-subadditive: if
$\Om = \bigcup \Om_j$ is a countable open cover, then
\begin{gather*}
 \|f\|^p_{p,w,\Om} \le \sum_j \|f\|^p_{p,w,\Om_j} \qquad \text{for every}\quad f \in L^p_w(\Om).
\end{gather*}
But it is not $\si$-additive.

\subsection{Sobolev spaces}
For $k \in \N$ and $1 \le p \le \infty$ we consider the Sobolev space
\begin{gather*}
 W^{k,p}(\Om) = \big\{f \in L^p(\Om) \colon \p^\al f \in L^p(\Om), \, 0 \le |\al| \le k\big\},
\end{gather*}
where $\p^\al f$ denote distributional derivatives, with the norm
\begin{gather*}
 \|f\|_{W^{k,p}(\Om)} := \sum_{|\al| \le k} \|\p^\al f\|_{L^p(\Om)}.
\end{gather*}
On bounded intervals $I \subseteq \R$ the Sobolev space $W^{1,1}(I)$
coincides with the space $AC(I)$ of~abso\-lutely continuous functions on~$I$
if we identify each $W^{1,1}$-function with its unique continuous representative.
Recall that a function $f \colon \Om \to \C$ on an open subset $\Om \subseteq \R$
is absolutely continuous ($AC$) if for every $\ep>0$ there exists $\de>0$ such that
for every finite collection of~non-overlapping intervals $(a_i,b_i)$, $i =1,\dots,n$, with
$[a_i,b_i] \subseteq \Om$ we have
\begin{gather*}
 \sum_{i=1}^n |a_i -b_i| < \de \implies \sum_{i=1}^n |f(a_i) -f(b_i)| < \ep.
\end{gather*}
Notice that $W^{1,\infty}(\Om) \cong C^{0,1}\big(\ol \Om\big)$ on Lipschitz domains (or more generally quasiconvex domains)~$\Om$.

We shall also use $W^{k,p}_{\on{loc}}$, $AC_{\on{loc}}$, etc.\ with the obvious meaning.

\subsection{Vector valued functions}

For our problem we need to consider mappings of Sobolev regularity with values in a finite-dimensional complex vector space $V$.
Let us fix a basis $v_1, \dots, v_n$ of $V$ and hence a linear isomorphism $\vh \colon V \to \C^n$. We~say that a mapping $f \colon \Om \to V$ is of Sobolev class $W^{k,p}$ if $\vh \o f$ is of class $W^{k,p}$.
The space $W^{k,p}(\Om, V)$ of all such mappings does not depend on the choice of the basis of $V$.

For $f= (f_1,\dots,f_n) \colon \Om \to \C^n$ we set
\begin{gather} \label{vectorvalued}
 \|f\|_{W^{k,p}(\Om,\C^n)} := \sum_{j=1}^n \|f_j\|_{W^{k,p}(\Om)}.
\end{gather}
If $f \in W^{k,p}(\Om, V)$, $f \ne 0$, and $\vh,\ps \colon V \to \C^n$ are two different basis isomorphisms, then
\begin{gather*}
 c \le \frac{\|\vh \o f\|_{W^{k,p}(\Om,\C^n)}}{\|\ps \o f\|_{W^{k,p}(\Om,\C^n)}} \le C
\end{gather*}
for positive constants $c,C>0$ which depend only on the linear isomorphism $\vh \o \ps^{-1}$. We~will denote by ${\|f\|_{W^{k,p}(\Om,V)}}$ or simply ${\|f\|_{W^{k,p}(\Om)}}$ any of the equivalent norms
$\|\vh \o f\|_{W^{k,p}(\Om,\C^n)}$.

Now suppose that we have a representation $\rh \colon G \to \on{GL}(V)$ of a finite group $G$ on
$V$. By~fixing a Hermitian inner product on $V$ and averaging it over $G$ we obtain a Hermitian inner product
with respect to which the action of~$G$ is unitary. We~could equivalently define
\begin{gather*}
 \|f\|_{W^{k,p}(\Om)} = \|f\|_{W^{k,p}(\Om,V)} := \sum_{|\al| \le k} \bigg(\int_{\Om} \|\p^\al f\|^p \, {\rm d}x\bigg)^{1/p},
 \end{gather*}
 where $\|\cdot\|$ is the norm associated with the $G$-invariant Hermitian inner product. In~that case $\|f\|_{W^{k,p}(\Om,V)}$ is $G$-invariant.

\subsection{Extension lemma}

The following extension lemma simply follows from the $\C$-valued version proved in \cite{ParusinskiRainer15}.
Similar versions can be found in \cite[Lemma 2.1]{ParusinskiRainerAC}
and \cite[Lemma 3.2]{GhisiGobbino13}.

\begin{Lemma} \label{lem:extend}
 Let $V$ be a finite-dimensional vector space.
 Let $\Om \subseteq \R$ be open and bounded, let $f \colon \Om \to V$ be continuous, $p \ge 1$,
 and set $\Om_0 := \{t \in \Om \colon f(t) \ne 0\}$.
 Assume that $f|_{\Om_0} \in AC_{\on{loc}}(\Om_0,V)$ and $f|_{\Om_0}' \in L^p(\Om_0,V)$.
 Then the distributional derivative of $f$ in $\Om$ is a measurable function $f' \in L^p(\Om,V)$ and
 \begin{gather*} 
 \|f'\|_{L^p(\Om,V)} = \|f|_{\Om_0}'\|_{L^p(\Om_0,V)},
 \end{gather*}
 where the $L^p$-norms are computed with respect to a fixed basis isomorphism.
\end{Lemma}

\section[Finite rotation groups in C]{Finite rotation groups in $\C$}\label{radicals}

Let $C_d \cong \Z/d\Z$ denote the cyclic group of order $d$ and consider its standard action on $\C$ by~rota\-tion.
Then $\C[\C]^{C_d}$ is generated by $\si(z) = z^d$. A lift over $\si$ of a function $f \colon \Om \to \C$ is a~solution of the
equation $z^d = f$.

The solution of the lifting problem in this simple example is completely understood. We~shall see that the general
solution is based on this prototypical case. Interestingly, it is also the case with the worst loss of regularity.

The following theorem is a consequence of a result of Ghisi and Gobbino \cite{GhisiGobbino13}.

\begin{Theorem} \label{cor:radicals}
 Let $d$ be a positive integer and
 let $I \subseteq \R$ be an open bounded interval. Assume that $f \colon I \to \C$ is a continuous function such that
 $f^d = g \in C^{d-1,1}\big(\oI\big)$.
 Then we have $f' \in L^{d'}_w(I)$ and
 \begin{gather} \label{est}
 \|f'\|_{d',w,I} \le
 C(d) \max\Big\{\big(\Lip_{I}\big(g^{(d-1)}\big)\big)^{1/d}|I|^{1/d'},
 \|g'\|_{L^\infty(I)}^{1/d}\Big\}.
 \end{gather}
\end{Theorem}

In other words
any continuous lift $f$ over $\si(z) = z^d$ of a curve in $C^{d-1,1}\big(\ol I,\si(\C)\big) = C^{d-1,1}\big(\ol I\big)$
is absolutely continuous and $f' \in L^{d'}_w(I)$ with the uniform bound~\eqref{est}.

\begin{Remark} \label{optimality}
This result is optimal: in general, $f'$ is not in $L^{d'}$ even if $g$ is real analytic (consider~$g(t) = t$).
On the other hand, if $g$ is only of class $C^{d-1,\be}\big(\overline I\big)$ for every $\be<1$, then $f$
does in~general not need to have bounded variation
in $I$ (see \cite[Example 4.4]{GhisiGobbino13}).
\end{Remark}

\begin{Remark}
 If we consider the \emph{real} representation of $C_d$ on $\R^2$ by rotation, basic invariants are given by
 \begin{gather*}
 \si_1(x,y)= z \ol z, \qquad
 \si_2(x,y)= \Re\big(z^d\big), \qquad
 \si_3(x,y) = \Im\big(z^d\big),\qquad \text{where}\quad
 z = x + {\rm i} y,
 \end{gather*}
 with the relation $\si_1^d = \si_2^2 + \si_3^2$.
 Let $f$ be a map that takes values in $\si\big(\R^2\big)$, where $\si=(\si_1,\si_2,\si_3)$, and which is smooth as a map
 into $\R^3$. Then the constraints $f$ has to fulfill, in contrast to the complex case where there are no constrains,
 give reasons for the more regular lifting in the real case (cf.\ Theorem~\ref{real}).

 For instance, suppose that $f$ is a smooth complex valued function. By~Theorem~\ref{real} and the previous paragraph,
 the equation $z^d = f$ has a solution of class $W^{1,\infty}$ provided that $|f|^{2/d}$ is of class $C^{d-1,1}$.
 Observe that for $d=2$ and $f\ge 0$ this condition is automatically fulfilled; it~corresponds to the hyperbolic case.
\end{Remark}

\section{Reduction to slice representations}

Let $G \acts V$ be a complex finite-dimensional representation of a finite group $G$.
Suppose that $\si=(\si_1,\dots,\si_n)$ is a system of homogeneous basic invariants.
Let $V^G = \{v \in V \colon Gv=v\}$ be the linear subspace of invariant vectors.
It is the subspace of all vectors $v$ for which the isotropy subgroup $G_v = \{g \in G \colon gv = v\}$ is
equal to $G$.

\subsection{Removing invariant vectors} 

Since finite groups are linearly reductive,
there exists a unique subrepresentation $V' \subseteq V$ such that $V = V^G \oplus V'$ (cf.\ \cite[Theorem 2.2.5]{DK02}).
Then
$\C[V]^G = \C\big[V^G\big] \otimes \C[V']^G$ and
$V/ G = V^G \times V' / G$.
A system of basic invariants of $\C[V]^G$ is given by a system of linear coordinates on $V^G$ together with a
system of basic invariants of $\C[V']^G$. Hence the following lemma is immediate.

\begin{Lemma} \label{fix}
Any lift $\overline f$ of a mapping $f=(f_0,f_1)$ in $V^G \times V' / G$ has the form
$\overline f = (f_0,\overline f_1)$,
where $\overline f_1$ is a lift of $f_1$.
\end{Lemma}

Consequently, we may assume without loss of generality that $V^G = \{0\}$.

\subsection{Luna's slice theorem}
Let us recall Luna's slice theorem.
Here we just assume that $V$ is a rational representation of~a~linearly reductive group $G$.
The categorical quotient $\pi \colon V \to V\cq G$ is the affine variety with~the coordinate ring $\C[V]^G$
together with the projection $\pi$ induced by the inclusion \mbox{$\C[V]^G \hookrightarrow \C[V]$}. In~this setting $\pi$ does not separate orbits, but for each element $z \in V \cq G$ there is a unique
closed orbit in the fiber $\pi^{-1}(z)$.
If $Gv$ is a closed orbit, then $G_v$ is again linearly reductive. We~say that $U \subseteq V$ is \emph{$G$-saturated} if $\pi^{-1} (\pi(U)) = U$.

\begin{Theorem}[{\cite{Luna73}, \cite[Theorem 5.3]{Schwarz80}}] \label{Luna}
Let $Gv$ be a closed orbit.
Choose a $G_v$-splitting \mbox{$T_v (Gv) \oplus N_v$} of $V \cong T_v V$ and let $\vh$ denote the mapping
\begin{gather*}
 G \times_{G_v} N_v \to V, \qquad [g,n] \mapsto g(v+n).
\end{gather*}
There is an affine open $G$-saturated subset $U$ of $V$ and an affine open $G_v$-saturated neighborhood~$B_v$ of $0$ in $N_v$
such that
\begin{gather*}
 \vh \colon\ G \times_{G_v} B_v \to U
\end{gather*}
and the induced mapping
\begin{gather*}
 \bar \vh \colon\ (G \times_{G_v} B_v)\cq G \to U\cq G
\end{gather*}
are \'etale.
Moreover, $\vh$ and the natural mapping $G \times_{G_v} B_v \to B_v \cq G_v$
induce a $G$-isomorphism of $G \times_{G_v} B_v$ with $U \times_{U\cq G} (B_v\cq G_v)$.
\end{Theorem}

\begin{Corollary}[{\cite{Luna73}, \cite[Corollary 5.4]{Schwarz80}}] \label{HST}
 In the setting of Theorem~$\ref{Luna}$,
 $G_y$ is conjugate to a subgroup of $G_v$ for all $y \in U$.	
 Choose a $G$-saturated neighborhood $\ol B_v$ of $0$ in $B_v$ $($classical topology$)$
 such that the canonical mapping $\ol B_v \cq G_v \to \ol U \cq G$ is a complex analytic isomorphism,
 where $\ol U = \pi^{-1}\big(\bar \vh((G \times_{G_v} B_v)\cq G)\big)$.
 Then $\ol U$ is a $G$-saturated neighborhood of $v$ and $\vh \colon G \times_{G_v} \ol B_v \to \ol U$
 is biholomorphic.
\end{Corollary}

\subsection{Uniform slice reduction} \label{ssec:red}

Let $\{\ta_i\}_{i=1}^m$ be a system of generators of $\C[N_v]^{G_v}$ and
let $\ta = (\ta_1,\dots,\ta_m) \colon N_v \rightarrow \mathbb{C}^m$ be the associated mapping.
Consider the \emph{slice}
\begin{gather*}
 S_v := v+\ol B_v,
\end{gather*}
where $\ol B_v$ is the neighborhood from Corollary~\ref{HST}.

\begin{Lemma} \label{lem:red}
 Let $a=(a_1,\dots,a_n)$ be a curve in $\si(V)$ with $a_k \ne 0$ and such that the curve
 \begin{gather*}
 \underline a := \big({a_k}^{-d_1/d_k} a_1,\dots,{a_k}^{-d_n/d_k} a_n\big)
 \end{gather*}
 lies in $\si(U_v)$, where $U_v$ is a neighborhood of $v$ in $S_v$.
 Composition of the curve $\underline a -\si(v)$ with the analytic isomorphism of Corollary~$\ref{HST}$ gives a
 curve $\underline b = (\underline b_1,\dots,\underline b_m)$ in $\ta(U_v-v)$ and
 \begin{gather*}
 b = (b_1,\dots,b_m)
 := \big(a_k^{e_1/d_k} \underline b_1,\dots,{a_k}^{e_m/d_k} \underline b_m\big),
 \qquad e_i = \deg \ta_i,
 \end{gather*}
 is a curve in $\ta(N_v)$. If $\overline b$ is a lift of $b$ over $\ta$ then
 \begin{gather*} 
 a_k^{1/d_k} v + \overline b
 \end{gather*}
 is a lift of $a$ over $\si$.
\end{Lemma}

\begin{proof}
 The curve $a_k^{-1/d_k} \overline b$ is a lift of $\underline b$ over $\ta$, indeed by homogeneity,
 \begin{gather*}
 \ta_i\big(a_k^{-1/d_k} \overline b\big) = a_k^{-e_i/d_k} \ta_i\big(\overline b\big)
 = a_k^{-e_i/d_k} b_i = \underline b_i.
 \end{gather*}
 Thus ${a_k}^{-1/d_k} \overline b + v$ is a lift of $\underline a$ over $\si$. By~homogeneity, we find $\si_i\big(\overline b + {a_k}^{1/d_k} v\big) = {a_k}^{d_i/d_k} \underline a_i = a_i$
 as required.
\end{proof}

The following lemma shows that the maximal degree of the basic invariants does not increase by
passing to a slice representation. It can be shown in analogy to \cite[Lemma 2.4]{KLMR06} or \cite{ParusinskiRainer14}.

\begin{Lemma} \label{lem:e}
 Assume that the systems of basic invariants $\{\si_j\}_{j=1}^n$ and $\{\ta_i\}_{i=1}^m$ are minimal and
 set $e := \max_{i} e_i = \max_{i} \deg \ta_i$. Then $e \le d$.
\end{Lemma}

In order to make the slice reduction uniform, we consider
the set
\begin{gather} \label{eq:compactK}
 K:= \bigg(\bigcup_{k=1}^n \big\{(\ul a_1, \dots,\ul a_n) \in \C^{n} \colon \ul a_k=1,\,|\ul a_j| \le 1 \text{ for }j \ne k\big\}\bigg) \cap \si(V),
\end{gather}
which is compact, since $\si(V)$ is closed. For each point $\underline p \in K$ choose $v \in \si^{-1}(\underline p)$.
Then the collection $\{\si(U_v)\}$ for all such $v$ is a cover of $K$ by sets $\si(U_v)$ that are open in the
trace topology on $\si(V)$
and on which
the conclusion of Lemma~\ref{lem:red} holds.
Choose a finite subcover
\begin{gather*} 
 \cB := \{B_\de\}_{\de \in \De} = \{\si(U_{v_\de})\}_{\de \in \De}.
\end{gather*}
Then there exists $\rh>0$ such that for every $\underline p \in K$ there is a $\de \in \De$ such that
\begin{gather} \label{eq:redball}
 B_\rh(\underline p) \cap \si(V) \subseteq B_{\de},
\end{gather}
where $B_\rh(\underline p)$ is the open ball with radius $\rh$ centered at $\ul p$.

\begin{Definition}
We refer to this data as the \emph{uniform slice reduction} of the representation~$G \acts V$,
in particular, we call $\rh>0$ from~\eqref{eq:redball} the \emph{uniform reduction radius}.
\end{Definition}

\section[Estimates for a curve in sigma(V)]{Estimates for a curve in $\boldsymbol{\si(V)}$} \label{aestimates}

In the next three sections we discuss preparatory lemmas for the proof of Theorem~\ref{main} which is then
given in Section~\ref{proof}.

\subsection{An interpolation inequality}
For an interval $I \subseteq \R$ and a function $f \colon I \to \C$ we set
\begin{gather*}
 V_I(f) := \sup_{t,s \in I} |f(t)-f(s)| = \on{diam} f(I).
\end{gather*}

\begin{Lemma}[{\cite[Lemma 4]{ParusinskiRainer15}}] \label{taylor}
 Let $I \subseteq \R$ be a bounded open interval, $m \in \N_{>0}$, and $\al \in (0,1]$.
 If $f\in C^{m,\al}(\overline I)$, then for all $t\in I$
 and $s = 1,\dots,m$,
 \begin{gather*}
 \big|f^{(s)}(t) \big| \le C |I|^{-s} \bigl(V_I(f) + V_I(f)^{(m+\al-s)/(m+\al)} \big(\Hoeld_{\al,I}\big(f^{(m)}\big)\big)^{s/(m+\al)} |I|^s
 \bigr),
 \end{gather*}
 for a universal constant $C$ depending only on $m$ and $\al$.
\end{Lemma}

\subsection{The local setup}

Let $G \acts V$ be a complex finite-dimensional representation of a finite group $G$.
Assume $V^G = \{0\}$. Let $\si=(\si_1,\dots,\si_n)$ be a system of homogeneous basic invariants
of degrees $d_1,\dots, d_n$ and let $d := \max_j d_j$.
Let $a \in C^{d-1,1}\big(\ol I,\si(V)\big)$, where $I \subseteq \R$ is a bounded open interval.

It will be crucial to consider the radicals $a_j^{1/d_j}$ of the components $a_j$ of $a$ which is justified by the following remark.

\begin{Remark} \label{rem:notation}
Every continuous selection $f$ of
the multi-valued function $a_j^{1/d_j}$ is absolutely continuous on~$I$, by Theorem~\ref{cor:radicals}.
(Clearly, continuous selections exist in this case.)
Moreover, $\|f'\|_{L^1(I)}$ is independent of the
choice of the selection. Indeed, if $g$ is a different continuous selection then on each connected component $J$ of
$I \setminus \{t \colon a_j(t)=0\}$ the functions $f$ and $g$ just differ by multiplication with a fixed $d_j$-th root of unity.
Thus $\|f'\|_{L^1(J)} = \|g'\|_{L^1(J)}$. The~$\C$-valued version of Lemma~\ref{lem:extend} implies that $\|f'\|_{L^1(I)} = \|g'\|_{L^1(I)}$.

Henceforth we fix one continuous selection of $a_j^{1/d_j}$ and
denote it by
\begin{gather*}
 \we a_j \colon\ I \to \C
\end{gather*}
as well as, abusing notation, by $a_j^{1/d_j}$. We~will also consider the absolutely continuous curve
\begin{gather*}
 \we a = (\we a_1,\dots,\we a_n) \colon\ I \to \C^n.
\end{gather*}
\end{Remark}

Suppose that $t_0 \in I$ and $k \in \{1,\dots,n\}$ are such that
\begin{gather} \label{eq:k}
 |\we a_k(t_0)| = \max_{1 \le j \le n} |\we a_j(t_0)| \ne 0.
 \end{gather}
Assume further that, for some fixed positive constant $B < 1/3$,
\begin{gather}\label{assumption}
 \|\we a'\|_{L^1 (I)} \le B |\we a_k(t_0)|.
\end{gather}
At this point we just demand that the constant $B$ is fixed and smaller than $1/3$; in the proof of Theorem~\ref{main}
we will additionally specify $B$ depending on the uniform reduction radius $\rh>0$ and the maximal degree $d$, see~\eqref{eq:constB},
and it is going to be fixed along said proof. In~accordance with~\eqref{vectorvalued}, $\|\we a'\|_{L^1 (I)} = \sum_{j=1}^n \|\we a_j'\|_{L^1 (I)}$.

\begin{Definition}
 By \emph{admissible data} for $G \acts V$ me mean a tuple $(a, I, t_0, k)$, where
 $a \in C^{d-1,1}(\ol I,\si(V))$ is a curve in $\si(V)$ for a representation $G \acts V$ with $V^G = \{0\}$
 defined on an open bounded interval $I$ such that
 $t_0 \in I$ and $k \in \{1,\dots,n\}$ satisfy~\eqref{eq:k} and~\eqref{assumption}.
\end{Definition}

\subsection[The reduced curve a]{The reduced curve $\boldsymbol{\ul a}$}

Let $(a, I, t_0, k)$ be admissible data for $G \acts V$. We~shall see in the next lemma that
$a_k$ does not vanish on the interval $I$ and so the curve
\begin{align}
 \underline a \colon\ I &\to \{(\underline a_1,\dots,\underline a_n) \in \C^{n} \colon \underline a_k=1\}, \notag
 \\
 t &\mapsto \ul a(t) := \big( a_k^{-d_1/d_k} a_1, \dots, a_k^{-d_n/d_k} a_n\big)(t)
 = \big( \big(\we a_k^{-1} \we a_1\big)^{d_1}, \dots, (\we a_k^{-1} \we a_n)^{d_n}\big)(t)
 \label{curve}
\end{align}
is well-defined. The homogeneity of the basic invariants implies that $\ul a(I) \subseteq \si(V)$.

\begin{Lemma} \label{lem1}
 Let $(a, I, t_0, k)$ be admissible data for $G \acts V$.
 Then for all $t \in I$ and $j=1,\dots,n$,
 \begin{gather} \label{eq:ass10}
 |\we a_j (t) - \we a_j (t_0)| \le B |\we a_k (t_0)|,
 \\[1ex]
 \frac 2 3 < 1-B \le \bigg | \frac{\we a_k(t)}{\we a_k(t_0)} \bigg | \le 1+B < \frac4 3,\label{eq:ass11}
 \\[1ex]
 |\we a_j(t)| \le \frac 4 3 |\we a_k(t_0)| \le 2 |\we a_k(t)|.\label{eq:ass12}
 \end{gather}
 The length of the curve $\ul a$
 is bounded by $3 d^2\, 2^{d} B$.
\end{Lemma}

\begin{proof}
 First~\eqref{eq:ass10} is a consequence of~\eqref{assumption},
 \begin{gather*}
 \big|\we a_j (t) - \we a_j (t_0)\big| = \bigg|\int_{t_0}^t \we a_j' \,{\rm d}s\bigg|
 \le \| \we a_j'\|_{L^1(I)} \le B |\we a_k(t_0)|.
 \end{gather*}
 Setting $j=k$ in~\eqref{eq:ass10} easily implies~\eqref{eq:ass11}. Together with~\eqref{eq:k}, the inequalities~\eqref{eq:ass10} and~\eqref{eq:ass11} give~\eqref{eq:ass12}. In~order to estimate the length of $\ul a$ observe that
 \begin{gather*}
 \ul a_j' = \p_t \big(\big(\we a_k^{-1} \we a_j\big)^{d_j} \big) = d_j \big(\we a_k^{-1} \we a_j\big)^{d_j-1} \big(\we a_k^{-1} \we a_j'
 - \we a_k^{-2} \we a_j \we a_k' \big).
 \end{gather*}
 Since $|\we a_k^{-1} \we a_j| \le 2$, by~\eqref{eq:ass12}, and thanks to~\eqref{eq:ass11} we obtain
 \begin{gather*}
 |\ul a_j'| \le 3 d\, 2^{d} |\we a_k(t_0)|^{-1} \big( |\we a_j'| + |\we a_k'|\big).
 \end{gather*}
 Consequently, using~\eqref{assumption},
 \begin{gather*}
 \int_I |\ul a'| \,{\rm d}s \le 3 d^2\, 2^{d} B,
 \end{gather*}
 as required.
\end{proof}

\section{The estimates after reduction to a slice representation} 

\subsection{The reduced local setup}

Let $(a, I, t_0, k)$ be admissible data for $G \acts V$ such that
for all
$j = 1,\dots,n$ and $ s = 1,\dots,d-1$,
 \begin{gather}%
 \big\|a_j^{(s)} \big\|_{L^\infty(I)} \le C(d) |I|^{-s} |\we a_k (t_0)|^{d_j},\nonumber
 \\
 \Lip_I \big(a_j^{(d-1)}\big) \le C(d) |I|^{-d} |\we a_k (t_0)|^{d_j}.
\label{est:a}
 \end{gather}
Here $C(d)$ is a positive constant which depends only on $d$; in
Lemma~\ref{lem:assThm1implyassProp3} we will see that the assumptions of Theorem~\ref{main}
imply~\eqref{est:a} on suitable intervals $I$, and the proof of Lemma~\ref{lem:assThm1implyassProp3}
will provide a specific value for $C(d)$.

Additionally,
we suppose that the curve $\underline a$ (defined in~\eqref{curve}) lies entirely in one of the balls~$B_\rh(\underline p)$ from
\eqref{eq:redball}. By~Lemma~\ref{lem:red}, we obtain a curve $b \in C^{d-1,1}\big(\ol I,\ta(W)\big)$, where $H \acts W$ with $H=G_v$ and $W=N_v$ is a
slice representation of $G \acts V$ and
\begin{gather} \label{eq:b_i}
 b_i = a_k^{e_i/d_k} \ps_i \big(a_k^{-d_1/d_k} a_1, \dots, a_k^{-d_n/d_k} a_n\big), \qquad i = 1,\dots,m,
\end{gather}
 where $e_i = \deg \ta_i$ and the $\ps_i$ are analytic functions which are bounded on their domain together with all
 their partial derivatives (this may be achieved by slightly shrinking the domain).

In accordance with Remark~\ref{rem:notation} we denote by
\begin{gather*}
 \we b_i \colon\ I \to \C
\end{gather*}
a fixed continuous selection of $b_i^{1/e_i}$. Sometimes it will also be convenient to use just the
symbol $b_i^{1/e_i}$ for $\we b_i$. We~set
\begin{gather*}
 \we b = \big(\we b_1,\dots,\we b_m\big) \colon\ I \to \C^m.
\end{gather*}
Hence~\eqref{eq:b_i} can also be written as
\begin{gather*}
 b_i = \we a_k^{e_i} \ps_i \big(\we a_k^{-d_1} a_1, \dots, \we a_k^{-d_n} a_n\big) = \we a_k^{e_i}\cdot \ps_i \o \ul a.
\end{gather*}

Thanks to Lemma~\ref{fix} we may assume that $W^H = \{0\}$.

\begin{Definition}
 By \emph{reduced admissible data} for $G \acts V$ me mean a tuple $(a, I, t_0, k; b)$, where
 $(a, I, t_0, k)$ is admissible data for $G \acts V$
 satisfying~\eqref{est:a}
 such that $\ul a$
 lies entirely in one of the balls~$B_\rh(\underline p)$ from~\eqref{eq:redball}
 and $b \in C^{d-1,1}(\ol I,\ta(W))$ is a curve resulting from Lemma~\ref{lem:red}
 and thus satisfies~\eqref{eq:b_i}.
\end{Definition}

The goal of this section is to show that the bounds~\eqref{est:a} are inherited by the curve $b$ on suitable subintervals.
This requires some preparation.

\subsection[Pointwise estimates for the derivatives of b on I]
{Pointwise estimates for the derivatives of $\boldsymbol b$ on~$\boldsymbol I$}
\begin{Lemma} 
 Let $(a, I, t_0, k; b)$ be reduced admissible data for $G \acts V$.
 Then for all
 $i=1,\dots,m$ and $s = 1,\dots,d-1$,
 \begin{gather}
 \big\| b_i^{(s)} \big\|_{L^\infty(I)} \le C |I|^{-s} |\we a_k (t_0)|^{e_i},\nonumber
 \\
 \Lip_I\big( b_i^{(d-1)}\big) \le C |I|^{-d} |\we a_k (t_0)|^{e_i},
\label{eq:b_ider}
 \end{gather}
 where $C$ is a constant depending only on $d$ and on the functions $\ps_i$.
\end{Lemma}

\begin{proof}
 Let us prove the first estimate in~\eqref{eq:b_ider}.
 Let $F$ be any $C^d$-function defined on an open set $U\subseteq \C^{n}$ that contains
 $\underline a(I)$ and assume $\|F\|_{C^d(\ol U)} < \infty$. We~claim that, for $s = 1,\dots,d-1$,
 \begin{gather} \label{eq:B2}
 \|\p_t^s (F \o \ul a)\|_{L^\infty(I)} \le C |I|^{-s},
 \end{gather}
 where
 $C$ is a constant depending only on $d$ and $\|F\|_{C^d(\ol U)}$.
 For any real exponent $r$, Fa\`a di Bruno's formula implies
 \begin{gather} \label{FaadiBruno}
 \p_t^{s} \big(a_j^{r} \big) =
 \sum_{\ell \ge 1}^{s} \sum_{\ga \in \Ga(\ell,s)} c_{\ga,\ell,r}
 \, a_j^{r-\ell} a_j^{(\ga_1)} \cdots a_j^{(\ga_\ell)},
 \end{gather}
 where $\Ga(\ell,s) = \{\ga \in \N_{>0}^\ell \colon |\ga| = s\}$ and
 \begin{gather*}
 c_{\ga,\ell,r}= \frac{s!}{\ell! \ga!} r(r-1) \cdots (r-\ell+1).
 \end{gather*}
 By~\eqref{est:a} and~\eqref{eq:ass11}, this implies for $j=k$
 \begin{align}
 \|\p_t^{s} \big( a_k^{r} \big)\|_{L^\infty(I)}
 &\le
 \sum_{\ell \ge 1}^{s} \sum_{\ga \in \Ga(\ell,s)} c_{\ga,\ell,r}
 \, \|a_k^{r-\ell}\|_{L^\infty(I)} \big\|a_k^{(\ga_1)}\big\|_{L^\infty(I)}
 \cdots \big\|a_k^{(\ga_\ell)}\big\|_{L^\infty(I)}
 \notag \\
 &\le
 C(d) \sum_{\ell \ge 1}^{s} \sum_{\ga \in \Ga(\ell,s)} c_{\ga,\ell,r}
 \, |a_k(t_0)|^{r-\ell}
 |I|^{-s} |a_k(t_0)|^{\ell}
 \notag \\
 &\le
 C(d)
 |I|^{-s} |a_k(t_0)|^{r }.\label{eq:derpower}
 \end{align}
 Together with the Leibniz formula,
 \begin{gather*}
 \p_t^s \big( a_k^{- d_j/d_k} a_j\big)
 = \sum_{q=0}^{s} \binom{s}{q} a_j^{(q)} \p_t^{s-q} \big( a_k^{- d_j/d_k} \big),
 \end{gather*}
~\eqref{eq:derpower} and~\eqref{est:a} lead to
 \begin{gather} \label{eq:underlinea}
\big \|\p_t^s \big(a_k^{- d_j/d_k} a_j\big) \big\|_{L^\infty(I)}
 \le C(d) |I|^{-s}.
 \end{gather}
 Again by the Leibniz formula,
 \begin{align*}
 \p_t (F \o \ul a)
 &= \sum_{j=1}^n ((\p_{j} F)\o \ul a)\, \p_t \big( a_k^{- d_j/d_k} a_j\big),
 \\
 \p_t^s (F \o \ul a)
 &= \sum_{j=1}^n \p_t^{s-1}\Big(((\p_{j} F)\o \ul a)\, \p_t \big( a_k^{- d_j/d_k} a_j\big)\Big)
 \\
 &= \sum_{j=1}^n \sum_{p=0}^{s-1}
 \binom{s-1}{p} \p_t^{p}((\p_{j} F)\o \ul a)\, \p_t^{s- p} \big( a_k^{- d_j/d_k} a_j\big).
 \end{align*}
 For $s=1$ we immediately get~\eqref{eq:B2}. For $1< s \le d-1$, we may argue by induction on $s$. By~induction hypothesis,
 \begin{gather*}
 \|\p_t^{p}((\p_{j} F)\o \ul a)\|_{L^\infty(I)} \le C\big(d,\|\p_{j} F\|_{C^{s}(\overline U)}\big) |I|^{-p},
 \end{gather*}
 for $p = 1,\dots,s-1$. Together with~\eqref{eq:underlinea} this entails~\eqref{eq:B2}.

 Now the first part of~\eqref{eq:b_ider} is a consequence of~\eqref{eq:b_i},~\eqref{eq:derpower} (for $r = e_i/d_k$),
 and~\eqref{eq:B2} (applied to $F= \ps_i$).

For the second part of~\eqref{eq:b_ider}
observe that for functions $f_1,\dots,f_m$ on~$I$ we have
\begin{gather*}
 \Lip_I(f_1 f_2 \cdots f_m) \le \sum_{i=1}^m \Lip_I(f_i) \|f_1\|_{L^{\infty}(I)} \cdots \widehat{\|f_i\|_{L^{\infty}(I)}}
 \cdots \|f_m\|_{L^{\infty}(I)}.
\end{gather*}
Applying it to~\eqref{FaadiBruno} and using
\begin{gather*}
 \Lip_I\big(a_j^{r-\ell}\big) \le |r-\ell| \|a_j^{r-\ell-1} \|_{L^\infty(I)} \| a_j'\|_{L^\infty(I)}
\end{gather*}
we find, as in the derivation of~\eqref{eq:derpower},
\begin{gather*}
 \Lip_I\big(\p_t^{d-1}( a_k^r)\big) \le C(d,r) |I|^{-d} |a_k(t_0)|^{r}.
\end{gather*}
As above this leads to
\begin{gather*}
 \Lip_I\big(\p_t^{d-1} \big( a_k^{- d_j/d_k} a_j\big) \big)
 \le C(d) |I|^{-d},
 \end{gather*}
and
\begin{gather*}
 \Lip_I \big(\p_t^{d-1}(F \o \ul a)\big) \le C\big(d, \|F\|_{C^d(\overline U)}\big) |I|^{-d},
\end{gather*}
and finally to the second part of~\eqref{eq:b_ider}.
\end{proof}

\subsection[Integral bounds for hat b']
{Integral bounds for $\boldsymbol{\we b'}$}

Recall that $e = \max_i e_i = \max_i \deg \ta_i$ and $e' = e/(e-1)$.

\begin{Corollary} \label{lem:Lpbak}
 Let $(a, I, t_0, k; b)$ be reduced admissible data for $G \acts V$.
 Then, for all $1 \le p < e'$ and all $i=1,\dots,m$,
 \begin{gather} \label{eq:Lpbak}
\big \|\we b_i'\big\|^*_{L^p(I)} \le C |I|^{-1} |\we a_k(t_0)|,
 \end{gather}
 for a constant $C$ which depends only on $d$, $p$, and the constant in~\eqref{eq:b_ider}.
\end{Corollary}

\begin{proof}
 Notice that, by Lemma~\ref{lem:e}, we have $e \le d$. By~\eqref{est} and~\eqref{eq:b_ider},
 \begin{align*}
 \big\| \we b_i'\big\|_{e_i',w,I}&= \big\| \big(b_i^{1/e_i}\big)'\big\|_{e_i',w,I}\le C(e_i) \max \Big\{\big(\on{Lip}_{I}\big( b_i^{(e_i-1)}\big)\big)^{1/e_i}|I|^{1/e_i'},
 \|b_i'\|_{L^\infty(I)}^{1/e_i}\Big\}
 \\
 &\le C |I|^{-1 +1/e_i'} |\we a_k (t_0)|,
\end{align*}
or equivalently,
\begin{gather*}
 \big\|\we b_i'\big\|^*_{e_i',w,I} \le C |I|^{-1} | \we a_k (t_0)|.
\end{gather*}
This entails~\eqref{eq:Lpbak} in view of~\eqref{eq:qp}.
\end{proof}

\subsection[Special subintervals of I and estimates on them]
{Special subintervals of $\boldsymbol I$ and estimates on them}\label{subintervals}

Let $(a, I, t_0, k; b)$ be reduced admissible data for $G \acts V$.

Suppose that $t_1 \in I$ and $\ell \in \{1,\dots,m\}$ are such that
\begin{gather} \label{eq:ell}
\big |\we b_\ell(t_1)\big| = \max_{1 \le i \le m}\big |\we b_i(t_1)\big| \ne 0.
\end{gather}
By~\eqref{eq:ass12} and~\eqref{eq:b_i}, for all $t \in I$ and $i = 1,\dots,m$,
 \begin{gather} \label{eq:b_ibound}
\big|\we b_i(t) \big| \le C_1 |\we a_k (t_0)|,
 \end{gather}
where the constant $C_1$ depends only on the functions $\ps_i$.
Thanks to~\eqref{eq:b_ibound} we can choose
a~con\-stant $D< 1/3$ and an open interval $J$ with $t_1 \in J \subseteq I$ such that
\begin{gather} \label{assumption2a}
 | J| |I|^{-1} {|\we a_k(t_0)|}
 + \big\|\we b '\big\|_{L^1 (J)} = D \big|\we b_\ell(t_1)\big|,
\end{gather}
where $\big\|\we b '\big\|_{L^1 (J)} = \sum_{i=1}^m\big \|\we b_i '\big\|_{L^1 (J)}$.
It suffices to take $D < C_1^{-1}$, where $C_1$ is the constant in~\eqref{eq:b_ibound}.
Here we use that $\we b _i$ is absolutely continuous,
by Theorem~\ref{cor:radicals}.

We will now see that on the interval $J$ the estimates of Section \ref{aestimates} hold for $b_i$
instead of $a_j$.

\begin{Lemma} \label{lemB}
 Let $(a, I, t_0, k; b)$ be reduced admissible data for $G \acts V$.
 Assume that $t_1 \in I$ and $\ell \in \{1,\dots,m\}$ are such that
~\eqref{eq:ell} holds and
 let $D$ and $J$ be as in~\eqref{assumption2a}.
 Then,
 for all $t \in J$ and $i = 1,\dots,m$,
 \begin{gather}
 \big| \we b_i (t) - \we b_i (t_1)\big| \le D \big|\we b_\ell (t_1)\big|, \label{b1}
 \\
 \frac 2 3 < 1-D \le \left | \frac{\we b_\ell(t)}{\we b_\ell(t_1)} \right| \le 1+D < \frac 4 3,
 \label{b2}
 \\
 \big|\we b_i(t)\big|\le \frac 4 3 \big|\we b_\ell(t_1)\big| \le 2 \big|\we b_\ell(t)\big|. \label{b3}
 \end{gather}
 The length of the curve
 \begin{gather*}
 J \ni t \mapsto \ul b(t)
 := \big( b_\ell^{-e_1/e_\ell} b_1, \dots, b_\ell^{-e_m/e_\ell} b_{m}\big)(t)
 = \big( \big(\we b_\ell^{-1} \we b_1\big)^{e_1}, \dots, \big(\we b_\ell^{-1} \we b_m\big)^{e_m}\big)(t)
 \end{gather*}
 in $\ta(W)$ is bounded by $3 e^2\, 2^{e} D$.
 For all $ i = 1,\dots,m$ and $ s = 1,\dots,d-1$,
 \begin{gather}
 \big\| b_i^{(s)} \big\|_{L^\infty(J)} \le C |J|^{-s} \big|\we b_\ell (t_1)\big|^{e_i},\nonumber
 \\
 \Lip_J\big(b_i^{(d-1)}\big) \le C |J|^{-d} \big|\we b_\ell (t_1)\big|^{e_i},\label{b4}
 \end{gather}
 for a universal constant $C$ depending only on $d$ and $\ps_i$.
\end{Lemma}

\begin{proof}
 The proof of~\eqref{b1}--\eqref{b3} is analogous to the proof of Lemma~\ref{lem1}; use~\eqref{eq:ell} and~\eqref{assumption2a}
 instead of~\eqref{eq:k} and~\eqref{assumption}. The bound for the length of the curve $J \ni t \mapsto \ul b(t)$ (which is
 well-defined by~\eqref{b2})
 follows from~\eqref{assumption2a}--\eqref{b3}; see the proof of Lemma~\ref{lem1}.

 Let us prove~\eqref{b4}. By~\eqref{eq:b_ider}, for $i = 1,\dots,m$ and $s= 1,\dots,d-1$
 \begin{gather}
 \big\| b_i^{(s)} \big\|_{L^\infty(I)} \le C |I|^{-s} |\we a_k (t_0)|^{e_i},\nonumber
 \\
 \Lip_I\big( b_i^{(d-1)}\big) \le C |I|^{-d} | \we a_k (t_0)|^{e_i},
 \label{eq:b_iderj}
 \end{gather}
 where $C = C(d, \ps_i)$. Recall that $e \le d$.

 For $s\ge e_i$ (including the case $s = d$), we have $\big(|J||I|^{-1}\big)^s \le \big(|J||I|^{-1}\big)^{e_i}$ and thus
 \begin{gather*}
 | I |^{-s} |\we a_k(t_0)|^{e_i}
 \le | J |^{-s} \big(|J| |I|^{-1} |\we a_k(t_0)| \big)^{e_i}
 \le | J |^{-s} \big|\we b_\ell(t_1)\big|^{e_i},
 \end{gather*}
 where the second inequality follows from~\eqref{assumption2a}. Hence~\eqref{eq:b_iderj} implies~\eqref{b4}.

 For $t\in J$ and $s< e_i$,
 \begin{align*}
 \big|b_i^{(s)}(t) \big|
 &\le C | J |^{-s} \bigl ( V_J(b_i)
 + V_J(b_i)^{(e_i-s)/e_i} \big(\Lip_J \big(b_i^{(e_i-1)}\big)\big)^{s/e_i} |J|^s \bigr)
 \hspace{7mm} \text{by Lemma~\ref{taylor}}
 \\
 &\le C_1 | J |^{-s} \Bigl ( \big|\we b_\ell(t_1)\big|^{e_i}
 + |\we b_\ell(t_1)|^{e_i-s} |J|^s |I|^{-s} |\we a_k(t_0)|^{s} \Bigr)
 \hspace{16mm} \text{by~\eqref{b3} and~\eqref{eq:b_iderj}}
 \\
 &\le C_2 | J |^{-s} \big|\we b_\ell(t_1)\big|^{e_i}, \hspace{69mm} \text{by~\eqref{assumption2a}}
 \end{align*}
 for constants $C = C(e_i)$ and $C_h = C_h(d,\ps_i)$.
\end{proof}

\section{A special cover by intervals} 

In the proof of Theorem~\ref{main} we shall have to glue local integral bounds on small intervals which result from the splitting
process to global bounds. In~this section we present a technical result which will
allow us to do so.

Let us suppose that $H \acts W$ is a complex finite-dimensional representation of a finite group~$H$,
$\ta= (\ta_1,\dots, \ta_m)$ is a system of homogeneous basic invariants of degree $e_i = \deg \ta_i$, and
$e := \max_i e_i$.

\subsection{Covers by prepared collections of intervals}\label{sec:preparation}
Let $I \subseteq \R$ be a bounded open interval and
let $b \in C^{e-1,1}\big(\overline I,\ta(W)\big)$.
For each point $t_1$ in
\begin{gather*}
 I' := I \setminus \{t \in I \colon b (t) = 0\}
\end{gather*}
there exists $\ell \in \{1, \dots, m\}$ such that~\eqref{eq:ell} holds.
Assume that there are positive constants $D < 1/3$ and $L$ such that
for all $t_1 \in I'$
there is
an open interval $J=J(t_1)$ with $t_1 \in J \subseteq I$ such that
\begin{gather} \label{assumption2aG}
 L | J| + \big\|\we b'\big\|_{L^1 (J)} = D\big |\we b_\ell(t_1)\big|.
\end{gather}
Note that~\eqref{eq:ell} and~\eqref{assumption2aG} imply~\eqref{b2} (cf.\ the proof of Lemma~\ref{lemB});
in particular, we have $J \subseteq I'$.

This defines a collection $\cI:=\{J(t_1)\}_{t_1\in I'}$ of open (in the relative topology) intervals which cover $I'$. We~will \emph{prepare} this collection in the following way.
Let us consider the functions
\begin{gather*}
 \vh_{t_1,+}(s) := L (s-t_1) + \big\|\we b'\big\|_{L^1 ([t_1,s))}, \qquad s\ge t_1,
 \\
 \vh_{t_1,-}(s) := L (t_1-s) + \big\|\we b'\big\|_{L^1 ((s,t_1])}, \qquad s\le t_1.
\end{gather*}
Then $\vh_{t_1,\pm} \ge 0$ are monotonic continuous functions defined for small $\pm (s - t_1)\ge 0$
and satisfying $\vh_{t_1,\pm}(t_1) = 0$.

Fix $t_1 \in I'$. Thanks to~\eqref{assumption2aG} there exist $s_-,s_+ \in \R$ such that
\begin{gather*}
\vh_{t_1,-}(s_-) + \vh_{t_1,+}(s_+) = D \big|\we b_\ell(t_1)\big|
\end{gather*}
and $J(t_1) = (s_-,s_+)$.
But there may also be a choice $s'_-,s'_+ \in \R$ such that this occurs \emph{symmetrically}, that is
\begin{gather*} 
 \vh_{t_1,-}(s'_-) = \vh_{t_1,+}(s'_+) = \frac D 2\big|\we b_\ell(t_1)\big|.
 \end{gather*}
If such a choice $s'_-,s'_+ \in \R$ exists, we replace $J(t_1)$ in the collection $\cI$ by the interval $(s'_-,s'_+)$.
(In \cite{ParusinskiRainer15} we said that these are intervals of \emph{first kind}.)
If such a choice does not exist, then we leave $J(t_1)$ in $\cI$ unchanged;
this happens when
we reach the boundary of the interval $I$ before either $\vh_{t_1,-}$ or $\vh_{t_1,+}$ has
 grown to the value $(D/2)|\we b_\ell(t_1)|$. (These intervals were said to be of~\emph{second kind} in \cite{ParusinskiRainer15}.)

If a collection $\cI$ satisfies this property, we say that it is \emph{prepared}.

\subsection{A special subcollection of intervals}

\begin{Proposition} \label{cover}
 Let $I \subseteq \R$ be a bounded open interval.
 Let $b \in C^{e-1,1}(\overline I,\ta(W))$.
 For each point $t_1$ in $I'$ fix $\ell \in \{1, \dots, m\}$ such that~\eqref{eq:ell} holds.
 Let $\cI=\{J(t_1)\}_{t_1\in I'}$ be a collection of~open intervals $J=J(t_1)$ with $t_1 \in J \subseteq I'$ such that:
\begin{enumerate}\itemsep=0pt
 \item[$1.$] There are positive constants $D < 1/3$ and $L$ such that
 for all $t_1 \in I'$ we have~\eqref{assumption2aG} for $J=J(t_1)$.
 \item[$2.$] The collection $\cI$ is prepared as explained in Section~$\ref{sec:preparation}$.
 \end{enumerate}
Then the collection $\cI$ has a countable subcollection $\cJ$ that still
 covers $I'$ and such that every point in $I'$ belongs to at most two intervals in $\cJ$. In~particular,
 \begin{gather*}
 \sum_{J \in \cJ} |J| \le 2 |I'|.
 \end{gather*}
\end{Proposition}

\begin{proof}
 It follows from the proof of \cite[Proposition 2]{ParusinskiRainer15}.
\end{proof}

\begin{Remark} \label{rem:cover}
 It is essential for us that $\cJ$ is a subcollection and not a refinement; by shrinking the intervals we would
 lose equality in~\eqref{assumption2aG}. We~will need this proposition for gluing local $L^p$-estimates to global ones.
\end{Remark}

\section{Proof of \texorpdfstring{Theorem~\ref{main}}{Theorem 1.1}} \label{proof}

The proof is based on uniform slice reduction and induction on the order of~$G$. We~will apply the following convention:
\begin{quote}
 \it We will no longer explicitly state all the dependencies of the constants. Henceforth, their dependence on the
 data of the uniform slice reductions will be subsumed by~simply indicating that they depend on the representation
 $G \acts V$. This includes the choice of $\si$: different choices of the basic invariants yield different constants.
 The~constants which are uniform in this sense will be denoted by $C = C(G \acts V)$ and may vary from line to line.
\end{quote}

\subsection*{Outline of the proof}
The proof of Theorem~\ref{main} is divided into three steps.
\begin{enumerate}\itemsep=0pt
\setlength{\leftskip}{0.45cm}
 \item[{\it Step} 1:] We check that for any $a \in C^{d-1,1}([\al,\be],\si(V))$ and all points $t_0 \in (\al,\be)$, where \mbox{$a(t_0) \ne 0$},
 we can find $k$ and a suitable interval $I$ such that $(a|_I,I,t_0,k;b)$, where $b$ is obtained by Lemma~\ref{lem:red},
 is reduced admissible data for $G \acts V$.
 \item[{\it Step} 2:] The reduced admissible data $(a|_I,I,t_0,k;b)$ represents the hypothesis of the inductive argument which is the heart of the proof.
 It will show that every continuous lift of $b$ is~absolutely continuous on~$I$ and it will give an $L^p$-bound for the first derivative
 of the lift on~$I$.
 \item[{\it Step} 3:] We assemble the proof of Theorem~\ref{main}.
 The local bounds will be glued to global bounds for lifts of the original curve $a$.
 \end{enumerate}

\subsection*{Step 1: The assumptions of Theorem~\ref{main} imply the local setup of the induction}

Assume that~$V^G = \{0\}$.
Let $a \in C^{d-1,1}([\al,\be],\si(V))$. Let $\rh$ be the uniform reduction radius from~\eqref{eq:redball}. We~fix a universal positive constant $B$ satisfying
\begin{gather} \label{eq:constB}
 B < \min\bigg\{\frac1 3, \frac{\rh}{3 d^2 2^{d}}\bigg\}.
\end{gather}
Fix $t_0 \in (\al,\be)$ and $k \in \{1,\dots,n\}$ such that
\begin{gather} \label{k}
 |\we a_k(t_0)| = \max_{1 \le j \le n} |\we a_j(t_0)|\ne 0.
\end{gather}
This is possible unless $a \equiv 0$ in which case nothing is to prove.
Choose a maximal open interval $I \subseteq (\al,\be)$ containing $t_0$ such that
\begin{gather}\label{assumption<=}
\const |I| + \|\we a'\|_{L^1 (I)} \le B |\we a_k(t_0)|,
\end{gather}
where
\begin{gather} \label{M}
 \const = \max_{1 \le j\le n} \big(\Lip_I\big(a_j^{(d-1)}\big)\big)^{1/d} |\we a_k (t_0)|^{(d-d_j)/ d}.
\end{gather}

Consider the point $\underline p = \ul a(t_0)$, where $\ul a$ is the curve defined in~\eqref{curve}. By~\eqref{k},
$\underline p$ is an~ele\-ment of the set $K$ defined in~\eqref{eq:compactK}. By~the properties of the uniform slice reduction specified in~Section~\ref{ssec:red},
the ball $B_\rh(\underline p)$ is contained in some ball of the finite cover $\cB$ of $K$. By~Lemma~\ref{lem1} and~\eqref{eq:constB},
the length of the curve $\ul a|_I$ is bounded by $\rh$.
Thus
\begin{gather} \label{bspecified}
 \text{$b \in C^{d-1,1}\big(\ol I,\ta(W)\big)$ is obtained by Lemma~\ref{lem:red} and satisfies~\eqref{eq:b_i}.}
\end{gather}

\begin{Lemma} \label{lem:assThm1implyassProp3}
	Assume that $V^G = \{0\}$.
 Let $(\al,\be) \subseteq \R$ be a bounded open interval and let $a \in C^{d-1,1}([\al,\be],\si(V))$.
 Let $B$ be a positive constant satisfying~\eqref{eq:constB}.
 Let $t_0 \in (\al,\be)$ and $k \in \{1,\dots,n\}$ be such that~\eqref{k} holds.
 Let $I$ be an open interval with $t_0 \in I \subseteq (\al,\be)$ satisfying~\eqref{assumption<=}
 and $b$ the reduced curve from~\eqref{bspecified}.
 Then $(a|_I,I,t_0,k;b)$ is reduced admissible data for $G \acts V$.
\end{Lemma}

\begin{proof}
	It remains to prove~\eqref{est:a}, i.e.,
	for all
	$ j = 1,\dots,n$ and $ s = 1,\dots,d-1$,
 \begin{gather*}
 \big\|a_j^{(s)} \big\|_{L^\infty(I)} \le C\, |I|^{-s} |\we a_k (t_0)|^{d_j},
 \\
 \Lip_I \big(a_j^{(d-1)}\big) \le C\, |I|^{-d} |\we a_k (t_0)|^{d_j},
 \end{gather*}	
 for $C = C(G \acts V)$.	
 The second bound is immediate from~\eqref{assumption<=}.
Let $t \in I$. By~Lemma~\ref{taylor},
 \begin{gather*}
 \big| a_j^{(s)}(t) \big| \le C |I|^{-s} \bigl( V_I(a_j)
 + V_I(a_j)^{(d-s)/d} \Lip_I\big(a_j^{(d-1)}\big)^{s/d} |I|^s\bigr).
 \end{gather*}
By~\eqref{eq:ass12} (it is clear that $(a|_I,I,t_0,k)$ is admissible data for $G \acts V$),
\begin{gather*}
 V_I(a_j) \le 2 \|a_j\|_{L^{\infty}(I)} \le 2\, (4/3)^d |\we a_k (t_0)|^{d_j},
\end{gather*}
and, by~\eqref{assumption<=},
\begin{gather*}
 \max_{1 \le j\le n} \big(\Lip_I\big(a_j^{(d-1)}\big)\big)^{s/d} |\we a_k (t_0)|^{-d_j s/d} |I|^s
 = |\we a_k (t_0)|^{-s} \const^s |I|^s \le 1.
\end{gather*}
Thus
\begin{gather*}
\big | a_j^{(s)}(t) \big| \le C |I|^{-s}
 |\we a_k (t_0)|^{d_j} \big(C_1 + C_2 \Lip_I\big(a_j^{(d-1)}\big)^{s/d} |\we a_k (t_0)|^{-d_j s/d} |I|^s \big)
\le C_3 |I|^{-s} |\we a_k (t_0)|^{d_j},
\end{gather*}
for constants $C_i$ that depend only on $d$.
So~\eqref{est:a} is proved.
\end{proof}

\subsection*{Step 2: The inductive argument}

The heart of the proof of Theorem~\ref{main} is the following

\begin{Proposition} \label{induction}
 Let $(a, I, t_0, k; b)$ be reduced admissible data for $G \acts V$.
 Then every con\-ti\-nuous lift $\ol b \in C^0(I,W)$ of $b$ is absolutely con\-ti\-nuous and satisfies
 \begin{gather} \label{muab}
 \|\ol b'\|_{L^p(I)} \le C \Big( \| |I|^{-1} {|\we a_k(t_0)|} \|_{L^p (I)}
 + \big\|\we b'\big\|_{L^p (I)}\Big),
 \end{gather}
 for all $1 \le p < d'$ and a constant $C$ depending only on $G \acts V$ and $p$.
\end{Proposition}

\begin{Remark}
 Notice that we bound the $L^p$-norm of the derivative of a general lift $\ol b$
 by the $L^p$-norm of the derivatives of the lifts $\we b_i$ for the standard action of rotation in $\C$.
\end{Remark}

\begin{proof}[Proof of Proposition~\ref{induction}]
	We proceed by induction on the group order.

\medskip\noindent
{\bf Induction basis.}
Proposition~\ref{induction} trivially holds, if the slice representation $H \acts W$
is trivial. In~that case $\C[W]^H \cong \C[W]$ and any system $\ta = (\ta,\dots,\ta_m)$ of linear coordinates
is a minimal system of generators. Hence $\ta \colon W \to \C^m$ is a linear isomorphism.
Moreover, $e_1 = e_2 = \cdots = e_m = e =1$ whence $\we b_i = b_i$ for all $i$.
Any lift of $b \in C^{d-1,1}\big(\ol I,\ta(W)\big)$ is of the form $\ol b = \ta^{-1} \o b$ and thus
\eqref{muab} is trivially satisfied.

\medskip\noindent
{\bf Inductive step.}
Let us set
\begin{gather*}
 I' := I \setminus \{t \in I \colon b(t) = 0\}.
\end{gather*}
For each $t_1 \in I'$ choose $\ell \in \{1, \dots, m\}$ such that~\eqref{eq:ell} holds. By~Section \ref{subintervals}, there is an open interval $J = J(t_1)$, $t_1 \in J \subseteq I'$,
such that~\eqref{assumption2a}, i.e.,
\begin{gather*}
 | J| |I|^{-1} {|\we a_k(t_0)|}
 + \big\|\we b'\big\|_{L^1 (J)} = D \big|\we b_\ell(t_1)\big| .
\end{gather*}
The constant $D$
can be chosen sufficiently small such that the length of the curve $\underline b|_J$ is bounded by
the uniform reduction radius $\si$ of the representation $H \acts W$.
It suffices to take
\begin{gather} \label{D}
 D < \min\bigg\{\frac 1 3, \frac{\si}{3 e^2 2^{e}}, C_1^{-1}\bigg\},
\end{gather}
where $C_1$ is the constant in~\eqref{eq:b_ibound}.
This follows from Lemma~\ref{lemB}
and the arguments in Section~\ref{ssec:red} and in Step~1 applied to $b$.

Then Lemma~\ref{lem:red} provides a curve $c \in C^{d-1,1}\big(\ol J, \pi(X)\big)$, where $K \acts X$ is a slice representation
of $H \acts W$, $\pi = (\pi_1,\dots,\pi_q)$ is a system of homogeneous basic invariants with degrees $f_1,\dots,f_q$,
and $f = \max_h f_h$.
The components of $c$ satisfy
\begin{gather*} 
 c_h = b_\ell^{f_h/e_\ell} \th_h \big(b_\ell^{-e_1/e_\ell} b_1, \dots, b_\ell^{-e_m/e_\ell} b_m\big), \qquad h = 1,\dots,q,
\end{gather*}
for suitable analytic functions $\th_h$. We~adopt our usual convention that
\begin{gather*}
\we c_h \colon\ J \to \C
\end{gather*}
denotes a fixed continuous selection of $c_h ^{1/f_h}$ and
set
\begin{gather*}
	\we c = (\we c_1 ,\dots,\we c_q) \colon\ J \to \C^q.
\end{gather*}
In view of Lemma~\ref{lemB} we conclude that $(b,J,t_1,\ell; c)$ is reduced admissible data for $H \acts W$.

By Proposition~\ref{cover} (where~\eqref{assumption2a} plays the role of~\eqref{assumption2aG}),
we may conclude that there is a countable family $\{(J_\ga,t_\ga,\ell_\ga,c_\ga)\}$ of open intervals $J_\ga \subseteq I'$,
of points $t_\ga \in J_\ga$,
of integers $\ell_\ga \in \{1, \dots,m\}$,
and reduced curves $c_\ga$ such that, for all $\ga$,
\begin{itemize}
	\item $(b,J_\ga,t_\ga,\ell_\ga;c_\ga)$ is reduced admissible data for $H \acts W$,
	\item we have
	\begin{gather}
 |J_\ga| |I|^{-1} {|\hat a_k(t_0)|} + \big\|\hat b'\big\|_{L^1 (J_\ga)} =
 D \big|\hat b_{\ell_\ga}(t_\ga)\big|, \label{Jba}
\end{gather}
\item and
\begin{gather}
	\bigcup_\ga J_\ga = I', \qquad \sum_\ga |J_\ga| \le 2 |I'|. \label{Jga}
\end{gather}
\end{itemize}

Let $\ol b \in C^0(I,W)$ be a continuous lift of $b$.
Fix $\ga$ and let $K \acts X$ be the corresponding slice representation of $H \acts W$.
Since $H$ is a finite group, we have $W \cong X$. With this identification and the decomposition $X = X^K \oplus X'$
we may deduce that the component of $\ol b$ in $X'$
is a~continuous lift of~$c_\ga$ on the interval~$J_\ga$.
To simplify the notation we will assume without loss of~generality that $X^{K} = \{0\}$ and that
$\ol b$ is a lift of $c_\ga$ on the interval $J_\ga$.

The induction hypothesis implies that
$\ol b$ is absolutely continuous on $J_\ga$ and satisfies
\begin{gather} \label{mu}
\big \|\ol b'\big\|_{L^p(J_\ga)} \le C \Big( \big\| |J_\ga|^{-1} {\big|\hat b_{\ell_\ga}(t_\ga)\big|} \big\|_{L^p (J_\ga)}
 + \|\hat c_{\ga}'\|_{L^p (J_\ga)}\Big),
\end{gather}
for all $1 \le p < e'$, where $C$ is a constant depending only on $H \acts W$ and $p$.

\medskip\noindent
{\bf $\boldsymbol{L^p}$-estimates on $\boldsymbol I$.}
To finish the proof of Proposition~\ref{induction} we have to show that the estimates~\eqref{mu}
on the subintervals $J_\ga$
imply the bound~\eqref{muab} on~$I$.
To this end we observe that Corollary~\ref{lem:Lpbak} (applied to $(b,J_\ga,t_\ga,\ell_\ga;c_\ga)$)
implies that,
for all $p$ with $1 \le p < f_\ga'$,
{\samepage\begin{gather} \label{eq:ctob}
 \|\hat c_{\ga}'\|^*_{L^p(J_\ga)}
 \le C |J_\ga|^{-1} \big|\hat b_{\ell_\ga} (t_\ga)\big|,
\end{gather}
for a constant $C$ that depends only on $H \acts W$ and $p$.

}

Now~\eqref{eq:ctob} and~\eqref{Jba} allow us to estimate the right-hand side of
\eqref{mu}:
\begin{align*}
 \big\||J_\ga|^{-1} \big|\hat b_{\ell_\ga}(t_\ga)\big| \big\|^*_{L^p (J_\ga)}
 + \|\hat c_{\ga}'\|^*_{L^p (J_\ga)}
 &= |J_\ga|^{-1} \big|\hat b_{\ell_\ga}(t_\ga)\big| +
 \|\hat c_{\ga}'\|^*_{L^p (J_\ga)}
 \\
 &\le C |J_\ga|^{-1} \big|\hat b_{\ell_\ga}(t_\ga)\big|
 \\
 &= C D^{-1} \Big( \big\| |I|^{-1} {|\hat a_k(t_0)|} \big\|^*_{L^1 (J_\ga)}
 + \big\|\hat b'\big\|^*_{L^1 (J_\ga)}\Big)
 \\
 &\le C D^{-1} \Big( \big\| |I|^{-1} {|\hat a_k(t_0)|} \big\|^*_{L^p (J_\ga)}
 + \big\|\hat b'\big\|^*_{L^p (J_\ga)}\Big)
\end{align*}
and therefore
\begin{gather}
\big \||J_\ga|^{-1} \big|\hat b_{\ell_\ga}(t_\ga)\big|\big \|^p_{L^p (J_\ga)}
 + \|\hat c_{\ga}'\|^p_{L^p (J_\ga)}
 \le C D^{-p} \Big( \big\| |I|^{-1} {|\hat a_k(t_0)|} \big\|^p_{L^p (J_\ga)}
 + \big\|\hat b'\big\|^p_{L^p (J_\ga)}\Big), \label{cba}
\end{gather}
for a constant $C$ that depends only on $H \acts W$ and $p$.

Let us now glue the bounds on $J_\ga$ to a bound on~$I$. By~\eqref{Jga},~\eqref{mu}, and~\eqref{cba},
\begin{gather*} 
 \sum_\ga \big\|\ol b'\big\|^p_{L^p(J_\ga)}
 \le C D^{-p} \Big( \big\| |I|^{-1} {|\hat a_k(t_0)|} \big\|^p_{L^p (I)}
 + \big\|\we b'\big\|^p_{L^p (I)}\Big),
\end{gather*}
for a constant $C$ that depends only on $H \acts W$ and $p$.
Thus $\ol b$ is absolutely continuous on $I'$
and
\begin{gather*}
 \big\|\ol b'\big\|_{L^p(I')}
 \le C D^{-1} \Big( \big\| |I|^{-1} {|\we a_k(t_0)|} \big\|_{L^p (I)}
 + \big\|\we b'\big\|_{L^p (I)}\Big),
\end{gather*}
for a constant $C$ that depends only on $H \acts W$ and $p$.
Since $\ol b$ vanishes on $I \setminus I'$,
Lemma~\ref{lem:extend} implies that $\ol b$ is absolutely continuous on~$I$ and
satisfies~\eqref{muab}, since $D = D(H \acts W)$ by~\eqref{D}.
This completes the proof of Proposition~\ref{induction}.
\end{proof}

\subsection*{Step 3: The proof of Theorem~\ref{main}}

In view of Lemma~\ref{fix} we may assume $V^G = \{0\}$.
Let $a \in C^{d-1,1}([\al,\be],\si(V))$. Suppose that $B$ is a positive constant fulfilling~\eqref{eq:constB} and
assume that
$t_0 \in (\al,\be)$, $k \in \{1,\dots,n\}$, and $I \ni t_0$
satis\-fy~\eqref{k} and~\eqref{assumption<=}.
Let $b \in C^{d-1,1}\big(\overline I, \ta(W)\big)$ be the reduced curve from~\eqref{bspecified}.
Then Lemma~\ref{lem:assThm1implyassProp3} implies that $(a,I,t_0,k;b)$ is reduced admissible data and
consequently each continuous lift $\ol b$ of~$b$ satisfies~\eqref{muab}, by Proposition~\ref{induction}. In~particular, if~$\ol a \in C^0((\al,\be),V)$ is a continuous lift of $a$, then we may assume
that $\ol a|_I$ is a lift of $b$.
It follows that $\ol a$ is absolutely continuous on~$I$ and
\begin{gather} \label{amuab}
 	\|\ol a'\|_{L^p(I)} \le C(G \acts V,p) \Big( \big\| |I|^{-1} {|\we a_k(t_0)|} \big\|_{L^p (I)}
 + \big\|\we b'\big\|_{L^p (I)}\Big).
 \end{gather}
Our next goal is to estimate the right-hand side of~\eqref{amuab} in terms of~$a$.

By Corollary~\ref{lem:Lpbak},
we get for all $p$ with $1 \le p < e'$,
\begin{gather} \label{beforecases}
 \big\||I|^{-1} |\we a_k (t_0)| \big\|^*_{L^p (I)} + \big\|\we b'\big\|^*_{L^p(I)}
 \le C |I|^{-1} |\we a_k (t_0)|,
\end{gather}
where the constant $C$ depends only on $G \acts V$ and $p$.
At this stage we distinguish the two cases of strict inequality or equality in~\eqref{assumption<=}:
\begin{enumerate}
 \item[$(i)$] Strict inequality: we have $I = (\al,\be)$ and
 \begin{gather}\label{assumption<}
 \const |I| + \|\we a'\|_{L^1 (I)} < B |\we a_k(t_0)|.
 \end{gather}
 \item[$(ii)$] Equality:
 \begin{gather}\label{assumption=}
 \const |I| + \|\we a'\|_{L^1 (I)} = B |\we a_k(t_0)|.
 \end{gather}
\end{enumerate}

\medskip\noindent
{\bf Case ($\boldsymbol i$).}
In this case we can reduce to the curve $b \in C^{d-1,1}\big(\overline I,\ta(W)\big)$ on the whole interval $I=(\al,\be)$; cf.\
Step 1. Thus,~\eqref{beforecases} becomes
\begin{gather*}
 \big\|(\be-\al)^{-1} |\we a_k (t_0)| \big\|_{L^p ((\al,\be))}
 + \big\|\we b'\big\|_{L^p((\al,\be))}
 \le C (\be-\al)^{-1+1/p} |\we a_k(t_0)|,
\end{gather*}
which can be bounded by
\begin{gather*} 
 C (\be-\al)^{-1+1/p} \max_{1 \le j \le n} \| a_j\|^{1/d_j}_{L^\infty((\al,\be))}.
\end{gather*}
By~\eqref{amuab}, $\ol a$ is absolutely continuous on $(\al,\be)$ and
\begin{gather} \label{lacase2}
 \|\ol a'\|_{L^p((\al,\be))} \le C (\be-\al)^{-1+1/p} \max_{1 \le j \le n} \|a_j\|^{1/d_j}_{L^\infty((\al,\be))},
\end{gather}
where $C = C(G \acts V,p)$.

\begin{Remark} 
The bound in~\eqref{lacase2} tends to infinity if $\be - \al \to 0$ unless $p=1$.
\end{Remark}

\medskip\noindent
{\bf Case ($\boldsymbol{ii}$).}
Using~\eqref{assumption=} to estimate~\eqref{beforecases} (as in the derivation of~\eqref{cba}),
we get
\begin{gather*}
 \||I|^{-1} |\we a_k (t_0)| \|_{L^p (I)} + \big\|\we b'\big\|_{L^p(I)}
 \le C \big(\const\|1\|_{L^p(I)} + \|\we a'\|_{L^p (I)}\big), 
\end{gather*}
for a constant $C$ that depends only on $G \acts V$ and $p$; note that $B=B(G \acts V)$ by~\eqref{eq:constB}.
Thus, by~\eqref{amuab},
\begin{gather*}
 \|\ol a'\|_{L^p(I)} \le C \big(\const \|1\|_{L^p(I)} + \|\we a'\|_{L^p (I)}\big).
\end{gather*}
Let us set $A : = \max_{1 \le j \le n} \| a_j\|^{1/d_j}_{C^{d-1,1}([\al,\be])}$. Then
\begin{gather*}
 \const = \max_{1 \le j\le n} \big(\Lip_I\big(a_j^{(d-1)}\big)\big)^{1/d} |\we a_k (t_0)|^{(d-d_j)/d}
 \le \max_{1 \le j\le n} A^{d_j/d} A^{(d-d_j)/d} = A.
\end{gather*}
Consequently,
\begin{gather} \label{eq:laIA}
 \|\ol a'\|_{L^p(I)} \le C \big(A \|1\|_{L^p(I)} + \|\we a'\|_{L^p (I)}\big).
\end{gather}

 By Proposition~\ref{cover} (applied to $a$ instead of $b$ and
~\eqref{assumption=} instead of~\eqref{assumption2aG}),
 we can cover the set
 $(\al,\be) \setminus \{t \colon a(t) = 0\}$
by a countable family $\cI$ of open intervals $I$ on which~\eqref{eq:laIA} holds and such that
$\sum_{I \in \cI} |I| \le 2 (\be - \al)$.
Together with Lemma~\ref{lem:extend} we may conclude that
$\ol a$ is absolutely continuous on $(\al,\be)$ and satisfies
\begin{gather*}
 \|\ol a'\|_{L^p((\al,\be))} \le C \big(A \|1\|_{L^p((\al,\be))}
 + \|\we a'\|_{L^p ((\al,\be))}\big),
\end{gather*}
Using~\eqref{est} and the fact that $1 - 1/d_j < 1/p$ for all $j \le n$, we obtain
\begin{gather*}
 \|\ol a'\|_{L^p((\al,\be))}
 \\ \qquad
{} \le C \bigg(A (\be - \al)^{1/p}
+ \sum_{j=1}^n \max \big\{\big (\Lip_{(\al,\be)}\big( a_j^{(d_j-1)}\big)\big)^{1/d_j} (\be -\al)^{1-1/d_j}, \|a_j'\|^{1/d_j}_{L^\infty((\al,\be))}
 \big\} \bigg)
 \\ \qquad
{}\le C\, \max\big\{1, (\be-\al)^{1/p} \big\} \max_{1 \le j \le n} \|a_j\|^{1/d_j}_{C^{d-1,1}([\al, \be])},
\end{gather*}
where $C = C(G \acts V,p)$.
The proof of Theorem~\ref{main} is complete.

\subsection*{Proof of Remark~\ref{rem:main}}

Remark~\ref{rem:main}($a$) is clear by the above discussion.

Suppose that there exists $s \in [\al,\be]$ such that $a(s) = 0$. Then for all $t \in (\al,\be)$ and all $j$,
\begin{gather*}
 |\we a_j (t)| = \bigg| \int_{s}^t \we a_j'(\ta)\,{\rm d}\ta\bigg| \le \|\we a_j'\|_{L^1((\al,\be))}.
\end{gather*}
Thus the Case ($i$), i.e.,~\eqref{assumption<}, cannot occur. This implies Remark~\ref{rem:main}($b$).

If the representation is coregular, then $\si(V) = \C^n$ and we may use a simple version of~Whitney's extension theorem to
extend $a$ to a curve defined on $(\al- 1,\be +1)$ which vanishes at the endpoints of this larger interval
and such that $\|a\|_{C^{d-1,1}([\al -1,\be +1])} \le C \|a\|_{C^{d-1,1}([\al,\be])}$, where $C$ is a
universal constant independent of $(\al,\be)$.
As above one sees that
Case ($i$) cannot occur and hence we obtain the bound~\eqref{eq:main} with the constant~\eqref{eq:betterbound}
on the larger interval $(\al- 1,\be +1)$. Thanks to the continuity of the extension, we obtain the desired bound
on the original inter\-val~$(\al,\be)$.
For details see \cite{ParusinskiRainer15}.
This shows Remark~\ref{rem:main}($c$). In~general, if $\si(V)$ is a proper subset of~$\C^n$, it is not clear
that the extended curve is contained in $\si(V)$ and hence liftable.

To see Remark~\ref{rem:main}($d$) we observe that under the assumption that $a^{(j)}(\al) = a^{(j)}(\be) = 0$ for all
$j= 1,\dots,d-1$ the curve $a$ can be extended beyond the interval $(\al,\be)$ by setting $a(t) = a(\al)$ for
$t<\al$ and $a(t) = a(\be)$ for $t>\be$. Then the extended curve still lies in $\si(V)$ and we have
$\|a_j\|_{C^{d-1,1}([\al -1,\be +1])} = \|a_j\|_{C^{d-1,1}([\al,\be])}$ for all $j =1,\dots,n$.
Choose a smooth function $\vh \colon \R \to [0,1]$ such that $\vh(t)=1$ for $t \le 0$ and $\vh(t)=0$ for $t\ge 1$.
Then $\ps(t):= \vh(\al-t)\vh(t-\be)$ is equal to $1$ on $[\al,\be]$ and $0$ outside $[\al-1,\be+1]$.
Let us consider
\begin{gather*}
 a_\ps := \big(\ps^{d_1} a_1,\ps^{d_2} a_2,\dots,\ps^{d_n} a_n\big),
\end{gather*}
which is a $C^{d-1,1}$-curve in $\si(V)$ coinciding with $a$ on $[\al,\be]$ and vanishing at the
endpoints of~$[\al-1,\be+1]$.
As above we conclude that for $a_\ps$ Case ($i$) cannot occur. Since we have
\begin{gather*}
 \|(a_\ps)_j\|_{C^{d-1,1}([\al -1,\be +1])} \le C(\vh) \|a_j\|_{C^{d-1,1}([\al,\be])}, \qquad
 j =1,\dots,n,
\end{gather*}
it is easy to conclude Remark~\ref{rem:main}($d$).

\section{Proof of \texorpdfstring{Theorem~\ref{main2}}{Theorem 1.3}} 

\begin{Lemma} \label{lem:1}
 Let $c \colon I \to \si(V)$ be continuous and let $\ol c \colon J \to V$ a continuous lift of $c$
 on an open proper subinterval $J \Subset I$ of $I$.
 Then $\ol c$ can be extended to a continuous lift of $c$ defined on~$I$.
\end{Lemma}

\begin{proof}
 Since we already know that $c$ admits a continuous lift $\ol c_1$ on~$I$ it suffices to show
 that $\ol c$ extends continuously to the endpoints of $J$. Then $\ol c$ can be extended left and
 right of $J$ by $\ol c_1$ after applying a fixed transformation from $G$.

 So let $t_0$ be the (say) right endpoint of $J$. The set of limit points $A$ of $\ol c(t)$ as $t \to t_0^-$
 is contained in the orbit corresponding to $c(t_0)$. On the other hand $A$ must be connected, by the continuity of $\ol c$.
 Since every orbit is finite, $A$ consists of just one point.
\end{proof}

\begin{Lemma} \label{lem:2}
	Let $\ol c_1$, $\ol c_2$ be continuous lifts of a curve $c \colon I \to \si(V)$.
	If $\ol c_1$ is absolutely continuous and $\ol c_1 \in W^{1,p}(I)$, then $\ol c_2$ is absolutely continuous, $\ol c_2 \in W^{1,p}(I)$, and
	\begin{gather*}
		\|\ol c_2' \|_{L^p(I)} \le C\, \|\ol c_1' \|_{L^p(I)},
	\end{gather*}
	where $C$ depends only on $G \acts V$ and on the coordinate system on $V$.
\end{Lemma}

\begin{proof}
	For each subset $E$ of $I$ we have $\ol c_2 (E) \subseteq \bigcup_{g \in G} g \ol c_1(E)$.
	It follows that
	\begin{gather*}
		\on{length}(\ol c_2) \le \sum_{g \in G} \on{length}(g \ol c_1) < \infty
	\end{gather*}
	and that $\ol c_2$ has the Luzin (N) property. Hence $\ol c_2$ is absolutely continuous.

	Suppose that both $\ol c_1$ and $\ol c_2$ are differentiable at $t$.
	After replacing $\ol c_1$ with $g \ol c_1$ for a suitable $g\in G$ we may suppose that $\ol c_1(t) = \ol c_2(t) =:v$.
	Then after switching to the slice representation at $v$ we have, for $g_h \in G_v$ (which entails $\ol c_2(t) = g_h \ol c_1(t)$),
	\begin{gather*}
		\frac{\ol c_2 (t+h) - \ol c_2(t)}{h} = \frac{g_h \ol c_1 (t+h) - g_h \ol c_1(t)}{h} = g_h \left(\frac{ \ol c_1 (t+h) - \ol c_1(t)}{h}\right),
	\end{gather*}
	which implies that $\ol c_2'(t) \in G_v \ol c_1'(t)$, since $G$, and hence $G_v$, is finite.
	This implies the lemma.
\end{proof}

Now we are ready to prove Theorem~\ref{main2}.

Let $\ol f \in C^0(U,V)$ be a continuous lift of $f \in C^{d-1,1}\big(\ol \Om,\si(V)\big)$ on $U$.

 By Theorem~\ref{main}, $\ol f$ is absolutely continuous along affine lines parallel to
 the coordinate axes (restricted to $U$).
 So $\ol f$ possesses partial derivatives of first order which are defined almost everywhere and measurable.

 Set $x=(t,y)$, where $t=x_1$, $y=(x_2,\dots,x_m)$,
 and let $U_1$ be the orthogonal projection of~$U$ on the hyperplane $\{x_1=0\}$.
 For each $y \in U_1$ we denote by $U^y := \{t \in \R \colon (t,y) \in U\}$ the corresponding section of $U$.

 Let $\ol f^y (t) := \ol f(t,y)$ for $t \in U^y$; it is clear that $\ol f^y$ is a continuous lift of $f|_{U^y \times \{y\}}$.
 Recall that $\Om = I_1 \times \cdots \times I_m$ is an open box in $\R^m$.
Let $\cC^y$ denote the set of connected components $J$ of the open subset $U^y \subseteq \R$.	 		
For each $J \in \cC^y$ we may extend the lift $\ol f^y$ continuously to $I_1 \times \{y\}$, by Lemma~\ref{lem:1}.
So for each $J \in \cC^y$ we get a continuous lift $\ol f^y_J$ of $f|_{I_1 \times \{y\}}$
such that $\ol f^y_J|_J = \ol f^y|_J$.

By Theorem~\ref{main}, for all $y \in U_1$ and $J \in \cC^y$,
 the lift $\ol f^y_J$ is absolutely continuous on $I_1$ with $\big(\ol f^y_J\big)' \in L^p(I_1)$, for $1 \le p < d/(d-1)$, and
 \begin{gather} \label{Aeq:ub}
 \big\|\big(\ol f^y_J\big)'\big\|_{L^p(I_1)} \le C \max_{1 \le i \le n} \|f_i\|^{1/d_i}_{C^{d-1,1}(\overline \Om)},
 \end{gather}
where $C$ depends only on $G \acts V$, $p$, and $|I_1|$.

Let $J,J_0 \in \cC^y$ be arbitrary. By~Lemma~\ref{lem:2}, both $\big(\ol f^y_J\big)'$ and $\big(\ol f^y_{J_0}\big)'$ belong to $L^p(I_1)$ and
\begin{gather*}
	\big\| \big(\ol f^y_J\big)' \big\|_{L^p(J)} \le C(G \acts V) \big\| \big(\ol f^y_{J_0}\big)' \big\|_{L^p(J)}.
\end{gather*}
Thus,
 \begin{gather*}
 \big\|\big(\ol f^y\big)'\big\|^p_{L^p(U^y)}
= \sum_{J \in \cC^y} \big\|\big(\ol f^y_J\big)'\big\|^p_{L^p(J)}
 \le C^p \sum_{J \in \cC^y} \big\|\big(\ol f^y_{J_0}\big)'\big\|^p_{L^p(J)}
 = C^p \big\|\big(\ol f^y_{J_0}\big)'\big\|^p_{L^p(U^y)}
 \end{gather*}
 and consequently, by~\eqref{Aeq:ub},
 \begin{gather*}
 	\big\|\big(\ol f^y\big)'\big\|_{L^p(U^y)} \le C \max_{1 \le i \le n} \|f_i\|^{1/d_i}_{C^{d-1,1}(\overline \Om)}.
 \end{gather*}
 By Fubini's theorem,
 \begin{gather*}
 \int_{U} \big|\p_1 \ol f(x)\big|^p\, {\rm dx} = \int_{U_1} \int_{U^y} \big|\p_1 \ol f(t,y)\big|^p\, {\rm d}t\, {\rm d}y
 \le \Big(C \, \max_{1 \le i \le n} \|f_i\|^{1/d_i}_{C^{d-1,1}(\overline \Om)} \Big)^p \int_{U_1} \, {\rm d}y.
 \end{gather*}
 This implies Theorem~\ref{main2}.

For Remark~\ref{rem:main2} notice that,
if $G \acts V$ is coregular, then $\si(V) = V \cq G = \C^n$ and hence we may use Whitney's extension theorem
to extend $f$ to a mapping defined on a box $R$ containing $\Om$ such that the $C^{d-1,1}$-norm on $R$ is bounded
by the $C^{d-1,1}$-norm on $\Om$ times a constant. In~general it is not clear that after extension $f$ still takes values in $\si(V)$.

\section[Q-valued functions]{$\boldsymbol Q$-valued functions} \label{sec:Q-valued}

The basic reference for the background on $Q$-valued Sobolev functions used in this section is \cite{De-LellisSpadaro11}.

\subsection{The metric space $\boldsymbol{\cA_Q(\R^n)}$}

Unordered $Q$-tuples of points in $\R^n$ can be formalized as positive atomic measures of mass $Q$.
Let $\llbracket p_i \rrbracket$ denote the Dirac mass at $p_i \in \R^n$. We~consider the space
\begin{gather*}
 \cA_Q(\R^n) := \bigg\{ \sum_{i = 1}^Q \llbracket p_i \rrbracket \colon p_i \in \R^n\bigg\}
\end{gather*}
of unordered $Q$-tuples of points in $\R^n$. Then $\cA_Q(\R^n)$ is a complete metric space when endowed with the metric
\begin{gather} \label{Qmetric}
 d\bigg(\sum_i \llbracket p_i \rrbracket, \sum_i \llbracket q_i \rrbracket\bigg)
 := \min_{\si \in \on{S}_Q} \bigg(\sum_i |p_i - q_{\si(i)}|^2\bigg)^{1/2}.
\end{gather}

\subsection{Invariants}

There is a natural one-to-one correspondence between
the unordered $Q$-tuples $\sum_i \llbracket p_i \rrbracket \in \cA_Q(\R^n)$
and the orbits of the $n$-fold direct sum $W := \big(\R^Q\big)^{\oplus n}$ of the tautological representation $\R^Q$ of~the
symmetric group $\on{S}_Q$. By~a result of Weyl \cite{Weyl39}, the algebra $\R[W]^{\on{S}_Q}$ is generated by the polarizations of the elementary symmetric functions.
Up to integer factors the polarizations are
\begin{gather*}
 \si_1(u) = \sum_i u_i,
 \\
 \si_2(u,v) = \sum_{i \ne j} u_i v_j,
 \\
 \si_3(u,v,w) = \sum_{i,j,k \text{ all } \ne } u_i v_j w_k,
 \\
 \cdots\cdots\cdots\cdots\cdots\cdots\cdots\cdots\cdots\cdots
 \\
 \si_Q(u,v,\dots,w) = \sum_{i,j,\dots,k \text{ all } \ne } u_i v_j\cdots w_k,
\end{gather*}
where $u = (u_1,u_2,\dots,u_Q)$, $v = (v_1,v_2,\dots,v_Q)$, etc.
A system of generators of $\R[W]^{\on{S}_Q}$ is obtained by substituting the arguments $x^1, x^2, \dots, x^n \in \R^Q$
for $u,v,w,\dots$ in all possible combinations (including repetitions).
Note that the ring $\R[W]^{\on{S}_Q}$ is not polynomial unless $n = 1$, e.g.,\ by the Shephard--Todd--Chevalley theorem.

\subsection{Subspaces $\boldsymbol{\cA_{G\acts \R^n}(\R^n)}$}

Let $G\acts \R^n$ be a representation of a finite group $G$. We~define the space
\begin{gather*}
 \cA_{G\acts \R^n}(\R^n) :=
 \bigg\{ \sum_{g \in G} \llbracket g p \rrbracket \colon p \in \R^n\bigg\}
\end{gather*}
of~$G$-orbits. It is a closed subspace of the complete metric space $\cA_{|G|}(\R^n)$, thus also complete.

A system of generators for $\R[V]^G$ can be obtained from the generators of $\R[W]^{\on{S}_{|G|}}$
by means of the Noether map $\eta^* \colon \R[W]^{\on{S}_{|G|}} \to \R[\R^n]^G$, where $\eta \colon \R^n \to W$
is defined by $\et(p)(g) = gp$ and $W= \big(\R^{|G|}\big)^{\oplus n}$ is identified with the space of mappings $G \to \R^n$;
for details see, e.g., \cite{Neusel:2002aa}.

\subsection{$\boldsymbol Q$-valued Sobolev functions}

Let $\Om$ be a bounded open subset of $\R^m$.
A measurable function $f \colon \Om \to \cA_Q(\R^n)$ is said to be in the Sobolev class $W^{1,p}$ (for $1 \le p \le \infty$)
if
\begin{enumerate}
 \item[1)] $x \mapsto {\rm d}(f(x),P) \in W^{1,p}(\Om)$ for all $P \in \cA_Q(\R^n)$,
 \item[2)] there exist functions $\vh_1,\dots,\vh_m \in L^p(\Om,\R^+)$ such that
 \begin{gather*}
 |\p_j {\rm d}(f,P)| \le \vh_j \qquad \text{a.e.\quad in}\quad \Om \qquad
 \text{for all}\quad P \in \cA_Q(\R^n)\quad \text{and}\quad j = 1,\dots,m.
 \end{gather*}
\end{enumerate}
The minimal functions $\vh_j$ satisfying (2) are denoted by $|\p_j f|$ and they are characterized as follows:
for every countable dense subset $\{P_\ell\}_{\ell \in \N} \subseteq \cA_Q(\R^n)$ and all $j = 1,\dots,m$
we have
\begin{gather*}
 |\p_j f| = \sup_{\ell \in \N} |\p_j {\rm d}(f,P_\ell)| \qquad \text{a.e.\quad in}\quad \Om.
\end{gather*}
One sets $|Df| := \big(\sum_{j= 1}^m |\p_j f|^2\big)^{1/2}$. This intrinsic approach is developed in \cite{De-LellisSpadaro11}.

Alternatively, one may use Almgren's extrinsic approach \cite{Almgren00} to $Q$-valued Sobolev functions.
There is an injective Lipschitz map $\xi \colon \cA_Q(\R^n) \to \R^N$, where $N = N(Q,n)$, with Lipschitz constant
$\on{Lip}(\xi) \le 1$ such that the inverse $\th:= \xi|_{\xi(\cA_Q(\R^n))}^{-1}$ is Lipschitz with Lipschitz constant
\mbox{$\le C(Q,n)$}.
Here the constants $N$ and $C$ depend only upon $Q$ and $n$.
The inverse $\th \colon \xi(\cA_Q(\R^n)) \to \cA_Q(\R^n)$ has a Lipschitz extension $\Th \colon \R^N \to \cA_Q(\R^n)$.
It follows that $\rh :=\xi \o \Th$ is a Lipschitz retraction of $\R^N$ onto $\xi(\cA_Q(\R^n))$.

A function $f \colon \Om \to \cA_Q(\R^n)$ is of class $W^{1,p}$ if and only if
$\xi \o f$ belongs to $W^{1,p}\big(\Om, \R^N\big)$, and~in~that case
\begin{gather} \label{eq:intrextr}
 |D(\xi \o f)| \le |Df| \le C(Q,n) |D(\xi \o f)|,
\end{gather}
see \cite[Theorem 2.4]{De-LellisSpadaro11}.

\subsection[Q-valued Sobolev functions and invariant theory]{$\boldsymbol Q$-valued Sobolev functions and invariant theory}

We may identify the $\on{S}_Q$-module $W = (\R^Q)^{\oplus n}$ with the space of $Q \times n$ matrices $\R^{Q \times n}$.
Then $\si \in \on{S}_Q$ acts on a $Q \times n$ matrix by permuting the rows. Consider the surjective mapping
$\pi \colon \R^{Q \times n} \to \cA_Q(\R^n)$ which sends a matrix with rows $p_1,\dots,p_Q$ to
$\sum_{i=1}^Q \llbracket p_i \rrbracket$. If we endow $\R^{Q \times n}$ with the Frobenius norm
$\big($i.e.,\ $\| (p_{ij})_{ij} \| = \big(\sum_{i=1}^Q \sum_{j=1}^n |p_{ij}|^2\big)^{1/2}\big)$ then $\pi$ is Lipschitz with $\on{Lip}(\pi) \le 1$.

Let $\si_1,\dots,\si_r$ be any system of homogeneous generators of $\R[W]^{\on{S}_Q}$. The corresponding map $\si=(\si_1,\dots,\si_r)$
induces a bijective map $\Si \colon \cA_Q(\R^n) \to \si(W) \subseteq \R^r$ such that $\si = \Si \o \pi$. We~may assume that $d_j := \deg \si_j \le Q$ for all $j = 1,\dots,r$.

\begin{Theorem} \label{thm:Q-valued}
 Let $\Om$ be a bounded open subset of $\R^m$.
 Let $f \colon \Om \to \cA_Q(\R^n)$ be con\-ti\-nuous.
 If $\Si \o f \in C^{Q-1,1}\big(\ol \Om,\R^r\big)$, then
 for each relatively compact open $\Om' \subseteq \Om$ we have
 $f \in W^{1,\infty}(\Om',\cA_Q(\R^n))$.
 Moreover,
 \begin{gather*}
 \|D f\|_{L^\infty(\Om')} \le C(Q,n,m,\Om,\Om') \Big(1 + \max_{1 \le j \le r} \|\Si_j \o f\|^{1/d_j}_{C^{Q-1,1}(\ol \Om)}\Big).
 \end{gather*}
\end{Theorem}

\begin{proof}
Let us first consider the case that $m =1$ and $\Om$ is an interval. In~that case we even obtain a global statement with $I :=\Om'= \Om$.
Indeed, the curve $c : = \Si \o f$ in $\si(W) \subseteq \R^r$ admits an absolutely continuous lift $\ol c$ to $W$ which belongs to
$W^{1,\infty}(I,W)$, by Theorem~\ref{real}. Then the statement follows by superposition with the Lipschitz map $\xi \o \pi$.
The uniform bound easily follows from the bound in Theorem~\ref{real} and~\eqref{eq:intrextr}:
 \begin{gather*}
\xymatrix{
 && W \ar@{->>}^{\pi}[d] \ar@{->>}^{\si}[drr] && \\
 I \ar^{f}[rr] \ar@{-->}^{\overline c}[rru] \ar@{..>}[rrd]
 && \cA_Q(\R^n) \ar@{->}^{\xi}[d] \ar@{^{(}->>}^{\Si}[rr] && \si(W) \subseteq \R^r. \\
 && \R^N & &
}
 \end{gather*}
 The general case follows from a standard argument by covering $\Om'$ by boxes contained in $\Om$ and
 using Fubini's theorem in a similar fashion as in the proof of Theorem~\ref{main2}.
\end{proof}

\begin{Corollary}
 The bijective mapping $\Si$ induces a bounded mapping
 \begin{gather*}
 \big(\Si^{-1}\big)_* \colon\ C^{Q-1,1}(\Om,\si(W)) \to W^{1,\infty}_{\on{loc}}(\Om,\cA_Q(\R^n)),\qquad \vh \mapsto \Si^{-1} \o \vh.
 \end{gather*}
\end{Corollary}

\begin{proof}
 It suffices to check that $\Si^{-1} \o \vh$ is continuous.
 This follows from the fact that $\pi$ is continuous and that $\si$ is proper and thus closed.
\end{proof}

\subsection{Multi-valued Sobolev liftings}

Let $G\acts \R^n$ be a representation of a finite group $G$.
The surjective map $\pi \colon \R^n \to \cA_{G \acts \R^n}(\R^n)$ defined by $\pi(p) = \sum_{g \in G} \llbracket gp \rrbracket$ is
clearly Lipschitz.
Let $\si_1,\dots,\si_r$ be any system of homogeneous generators of $\R[\R^n]^{G}$.
There is a bijective map $\Si \colon \cA_{G \acts \R^n}(\R^n) \to \si(\R^n) \subseteq \R^r$ such that $\si = \Si \o \pi$, since
$\si=(\si_1,\dots,\si_r)$ separates orbits.
Let $d := \max_j \deg \si_j$.

Let $\Om$ be a bounded open subset of $\R^m$. We~say that a function $f \colon \Om \to \cA_{G \acts \R^n}(\R^n)$ is~of~class~$W^{1,p}$,
and write $f \in W^{1,p}(\Om,\cA_{G \acts \R^n}(\R^n))$, if $f \in W^{1,p}(\Om,\cA_{|G|}(\R^n))$.

Thus we obtain, analogously to Theorem~\ref{thm:Q-valued},

\begin{Theorem}
 Let $f \colon \Om \to \cA_{G \acts \R^n}(\R^n)$ be continuous.
 If $\Si \o f \in C^{d-1,1}\big(\ol \Om,\R^r\big)$, then
 for each relatively compact open $\Om' \subseteq \Om$ we have
 $f \in W^{1,\infty}(\Om',\cA_{G \acts \R^n}(\R^n))$.
 Moreover,
 \begin{gather*}
 	\|D f\|_{L^\infty(\Om')} \le C(d,n,m,\Om,\Om') \Big(1 + \max_{1 \le j \le r} \|\Si_j \o f\|^{1/d_j}_{C^{d-1,1}(\ol \Om)}\Big).
 \end{gather*}
\end{Theorem}

\begin{Corollary}
 The bijective mapping $\Si$ induces a bounded mapping
 \begin{gather*}
 \big(\Si^{-1}\big)_* \colon\ C^{d-1,1}(\Om,\si(\R^n)) \to W^{1,\infty}_{\on{loc}}(\Om,\cA_{G \acts \R^n}(\R^n)),\qquad \vh \mapsto \Si^{-1} \o \vh.
 \end{gather*}
\end{Corollary}

\subsection[Complex Q-valued functions]{Complex $\boldsymbol Q$-valued functions}

It is evident that one can define the space $\cA_Q(\C^n)$ of unordered $Q$-tuples of points in $\C^n$ in~analogy to $\cA_Q(\R^n)$.
It is a complete metric space with the metric $d$ from~\eqref{Qmetric}.
Again there is a natural bijection between the points in $\cA_Q(\C^n)$ and the orbits of the $\on{S}_Q$-module $\big(\C^Q\big)^{\oplus n}$,
the basic invariants of which are again given by the polarizations of the elementary symmetric functions.

Given a complex representation $G \acts \C^n$ of a finite group $G$ we may consider the closed subspace $\cA_{G \acts \C^n}(\C^n)$
of $\cA_{|G|}(\C^n)$.

The theory of complex $Q$-valued Sobolev functions can simply be taken over from the identification
$\cA_Q(\C^n) \cong \cA_Q\big(\R^{2n}\big)$ induced by $\C \cong \R^2$.

Let $\Om$ be a bounded open subset of~$\R^m$.
With the analogous definition of the basic invariants~$\si_i$ and the maps $\pi$ and $\Si$ we may deduce from Theorem~\ref{main} the following

\begin{Theorem}
 Let $f \colon \Om \to \cA_Q(\C^n)$ be continuous.
 If $\Si \o f \in C^{Q-1,1}\big(\ol \Om,\C^r\big)$, then
 for each relatively compact open $\Om' \subseteq \Om$ and all $1 \le p < Q/(Q-1)$ we have
 $f \in W^{1,p}(\Om',\cA_Q(\C^n))$.
 Moreover,
 \begin{gather*}
 \|D f\|_{L^p(\Om')} \le C(Q,n,m,p,\Om,\Om') \Big(1 + \max_{1 \le j \le r} \|\Si_j \o f\|^{1/d_j}_{C^{Q-1,1}(\ol \Om)}\Big).
 \end{gather*}
\end{Theorem}

Similarly we get

\begin{Theorem}
 Let $f\colon \Om \to \cA_{G \acts \C^n}(\C^n)$ be continuous.
 If $\Si \o f \in C^{d-1,1}\big(\ol \Om,\C^r\big)$, then
 for each relatively compact open $\Om' \subseteq \Om$ and all $1 \le p < d/(d-1)$ we have
 $f \in W^{1,p}(\Om',\cA_{G \acts \C^n}(\C^n))$.
 Moreover,
 \begin{gather*}
 \|D f\|_{L^p(\Om')} \le C(d,n,m,p,\Om,\Om') \Big(1 + \max_{1 \le j \le r} \|\Si_j \o f\|^{1/d_j}_{C^{d-1,1}(\ol \Om)}\Big).
 \end{gather*}
\end{Theorem}

Again we may conclude that
the bijective mapping $\Si$ induces a bounded mapping
 \begin{gather*}
 \big(\Si^{-1}\big)_* \colon\ C^{Q-1,1}\big(\Om,\si\big(\big(\C^Q\big)^{\oplus n}\big)\big) \to W^{1,p}_{\on{loc}}(\Om,\cA_Q(\C^n)),\qquad \vh \mapsto \Si^{-1} \o \vh,
 \end{gather*}
 for all $1 \le p <Q/(Q-1)$. In~the case of a $G$-module $\C^n$ we find that
 \begin{gather*}
 \big(\Si^{-1}\big)_* \colon\ C^{d-1,1}(\Om,\si(\C^n)) \to W^{1,p}_{\on{loc}}(\Om,\cA_{G \acts \C^n}(\C^n)),\qquad \vh \mapsto \Si^{-1} \o \vh,
 \end{gather*}
 is a bounded mapping for all $1 \le p <d/(d-1)$.

\subsection*{Acknowledgements}
Supported by the Austrian Science Fund (FWF), Grant P~32905-N and START Programme Y963, and by ANR project ANR-17-CE40-0023 - LISA.

\pdfbookmark[1]{References}{ref}
\LastPageEnding

\end{document}